\documentclass[arxiv,reqno]{ml-gen}
\usepackage{amsmath,amscd,upref} 
\usepackage{amssymb,amsfonts,mathrsfs}
\usepackage{xy} \xyoption{all}
\usepackage[symbol]{footmisc}
\usepackage{perpage}
\MakePerPage[5]{footnote}
\usepackage[colorlinks=true,linkcolor=blue,citecolor=blue,backref=page]{hyperref}
\usepackage{math60,mlthm}
\usepackage{local}

\usepackage{euler}

\begin{document}
\title[Spectral flow and the unbounded Kasparov product]{Spectral flow and the
  unbounded Kasparov product}

\author{Jens Kaad}
\address{Hausdorff Center for Mathematics,
Universit\"at Bonn,
Endenicher Allee 60,
53115 Bonn,
Germany}
\email{jenskaad@hotmail.com}

\author{Matthias Lesch}
\address{Mathematisches Institut,
Universit\"at Bonn,
Endenicher Allee 60,
53115 Bonn,
Germany}

\email{ml@matthiaslesch.de, lesch@math.uni-bonn.de}
\urladdr{www.matthiaslesch.de, www.math.uni-bonn.de/people/lesch}
\thanks{Both authors were supported by the 
Hausdorff Center for Mathematics, Bonn}

\subjclass[2010]{Primary 19K35; Secondary 46H25, 58J30, 46L80}
\keywords{$KK$-theory, operator modules, Kasparov product, spectral flow}
\begin{abstract}
We present a fairly general construction of unbounded representatives for
the interior Kasparov product.  As a main
tool we develop a theory of $C^1$-connections on operator $*$ modules; we do
not require any smoothness assumptions; our $\sigma$--unitality assumptions
are minimal. Furthermore, we use work of Kucerovsky
and our recent Local Global Principle for regular operators in Hilbert
$C^*$--modules.

As an application we show that the Spectral Flow Theorem 
and more generally the index theory of Dirac-Schr\"odinger operators can be
nicely explained in terms of the interior Kasparov product.
\end{abstract}

\maketitle
\tableofcontents

\section{Introduction}\label{s:intro}     
The Spectral Flow Theorem (\cite{RobSal:SFM}, or quite recently
\cite{GesEtAl:IFS}) relates the spectral flow of a family, $A(x)$, of
unbounded selfadjoint Fredholm operators to the index of the Fredholm operator
$D=\frac{d}{dx} + A(x)$. $D$ is an example of a so called Dirac-Schr\"odinger
operator on the complete manifold $\R$. Index theorems for such operators, at
least in the special case where $A(x)$ is a finite rank bundle morphism, were
established in the 80s and 90s, \eg Anghel \cite{Ang:ICO}, \cite{Ang:AIT}
or \cite[Chap.~IV]{Les:OFT} and the references therein.

The family $\{A(x)\}_{x \in \rr}$ naturally defines a class $[F_1]$ in the
first $K$-theory group of $C_0(\rr)$, while the Dirac-operator
$-i\frac{d}{dx}$ defines a class $[F_2]$ in the first $K$-homology group of
the same $C^*$-algebra. It follows from $KK$-theory that the spectral flow of
$A(x)$ can be recovered as the Kasparov product of the classes $[F_1]$ and
$[F_2]$ via the natural identification $KK(\cc,\cc) \cong \zz$, \cf
\cite[Sec.~18.10]{Bla:KTO2Ed}. On the other hand classes in $KK(\cc,\cc)$ 
are represented by Fredholm operators and therefore the Spectral Flow Theorem
can be rephrased in the following way: The Dirac-Schr\"odinger operator 
$D = \frac{d}{dx} + A(x)$, viewed as an unbounded $\cc$-$\cc$ Kasparov module
represents the interior Kasparov product of $[F_1] \in K_1(C_0(\rr))$ and
$[F_2] \in K^1(C_0(\rr))$. 

It is tempting to generalize this pattern by replacing the real line by a
complete Riemannian manifold. The family then becomes parametrized by the
manifold whereas a Dirac operator (or, slightly more generally, a first order elliptic
operator with bounded propagation speed) on the complete manifold naturally replaces
$-i\frac{d}{dx}$. For the realization of this program it turns out that the
existing theories of unbounded representatives for the $KK$-product, see \eg
\cite{BaaJul:TBK}, \cite{Kuc:KKP},  \cite{Mes:UBK}), do not suffice. It is the purpose of
this paper to establish an appropriate improvement of unbounded $KK$-theory
which naturally covers Dirac-Schr\"odinger operators on complete manifolds. 

In the paper \cite{Mes:UBK} Mesland develops a framework of smooth algebras and
differentiable $C^*$-modules equipped with smooth connections with the purpose
of establishing a general formula for the unbounded $KK$-product. We pursue a
less technical approach: Since unbounded Kasparov modules are abstractions of
\emph{first order} elliptic differential operators it is most natural to
impose $C^1$-conditions on both modules and connections. It turns out that the
theory of operator modules and complete boundedness provides a good operator
algebraic framework for treating such concepts.

More concretely, let us fix a pair of unbounded (odd) Kasparov modules
$(X,D_1)$ and $(Y,D_2)$ over $C^*$-algebras $A$-$B$ and $B$-$C$, respectively.
The $C^*$-algebra $B$ then possesses a dense operator $*$-algebra (\cf Section
\ref{s:OStarM}) $B_1$ which is the largest algebra for which the unbounded
derivation defined by $D_2$ yields bounded adjointable operators. This
operator $*$-algebra defines a $C^1$-structure on the $C^*$-algebra $B$. By
Kasparov's stabilization theorem \cite{Kas:HCM},
\cite[Sec.~13.6.2]{Bla:KTO2Ed}, \cite[Sec.~6.2]{Lan:HCM} the Hilbert
$C^*$-module $X$ is a direct summand in the standard module $B^\infty$ over
$B$. Let $P \in \B (B^\infty)$ denote the projection with $PB^\infty \cong X$.
We say that $X$ has a $C^1$-structure if the projection $P$ descends to a
\emph{completely bounded projection} (see \eg \cite{ChrSin:RCB}) on the
standard module $B_1^\infty$ over the operator $*$-algebra $B_1$. The image
$PB_1^\infty \su PB^\infty$ is an operator $*$-module over $B_1$. A
$C^1$-structure on $X$ gives rise to a Gra{\ss}mann $D_2$-connection
$\Gc_{D_2}$ (Def.  \ref{def:LCConn}, \cf \cite[p.~600]{Con:CAG}) which is an
essential ingredient for the construction of the unbounded
$KK$-product.

To formulate our main result we need to introduce one more technical device,
namely that of a correspondence between $(X,D_1)$ and $(Y,D_2)$ (Def.
\ref{def:Correspondence}). Roughly speaking a correspondence from $(X,D_1)$ to $(Y,D_2)$ is a pair
$(X_1,\Na^0)$ consisting of the operator $*$--module $PB_1^\infty$ over $B_1$
and a hermitian $D_2$--connection $\Na^0_{D_2} : PB_1^\infty \to X \hot_B
\B(Y)$ such that
\begin{enumerate}
 \item\label{I:itemA} The commutator $[1 \ot_{\Na^0} D_2,a] : \sD(1 \ot_{\Na^0} D_2) \to X
  \hot_B Y$ is well--defined and extends to a bounded operator on $X \hot_B Y$
  for all $a \in A_1$.
 \item\label{I:itemB} The unbounded operator
\[
[D_1\ot 1, 1 \ot_{\Na^0} D_2](D_1 \ot 1 - i \cd \mu)^{-1} : \sD(1 \ot_{\Na^0}
D_2) \to X \hot_B Y
\]
is well--defined and extends to a bounded operator on $X \hot_B Y$, for
all $\mu\in\Rstar$.
\end{enumerate}\mpar{I think I can prove that it suffices to check $\mu=\pm
i$, but, similarly to the deficiency indices, $i$ or $-i$ alone does not
suffice}

Note that \eqref{I:itemB} is less restrictive than the commutator conditions
imposed by Mesland in \cite[Definitions 4.9.1 and 4.9.5]{Mes:UBK}; the latter
in particular imply the boundedness of the commutator 
$[D_1 \ot 1,1 \ot_{\Na^0} D_2]$. Our weaker condition of \emph{relative
boundedness} of the commutator is already needed for the Spectral Flow Theorem
over the real line in the context of \cite{RobSal:SFM}.

The main result of this paper can then be stated as follows:
\begin{theorem}\label{I:ThmA}
Let $(X,D_1)$ and $(Y,D_2)$ be two odd unbounded Kasparov modules for $(A,B)$
and $(B,C)$ respectively. Suppose that there exists a correspondence
$(X_1,\Na^0)$ from $(X,D_1)$ to $(Y,D_2)$. Let 
$\Na_{D_2} : X_1 \to X \hot_B \B(Y)$ be any hermitian $D_2$--connection. Then the pair 
$(D_1 \times_{\Na} D_2, (X \hot_B Y)^2)$ is an even unbounded Kasparov $A$--$C$
module which represents the interior Kasparov product of $(X,D_1)$ and 
$(Y,D_2)$. 
\end{theorem}
Here $D_1 \times_{\Na} D_2$ is essentially the operator $D_1\ot 1 \pm i\, 1\ot_\Na
D_2$, see \Eqref{eq:UnbProdOp} below.
Theorem \ref{I:ThmA} is a combination of Theorem \ref{t:prodmod} and Theorem
\ref{t:prodcoincide}.

Let us briefly outline the above mentioned application to Dirac-Schr\"odinger
operators.  Let $M$ be a complete oriented manifold of dimension $m$ and
$\{D_1(x)\}_{x \in M}$ a family of unbounded selfadjoint operators
parametrized by the manifold. We assume that the domain $W:=\dom(D_1(x))$ is
independent of $x$ and that the graph norms of the family are uniformly
equivalent.  Furthermore, the map $D_1 : M \to \B(W,H)$ is assumed to be
weakly differentiable with uniformly bounded derivative; see Subsection
\ref{ss:exdirschr} for the precise formulation. On top of these conditions we
will require that the inclusion $\io: W \to H$ is \emph{compact} and that the
spectra of $D_1(x), x\in M, $ are uniformly bounded away from zero outside a
compact set $K \su M$. These conditions are essentially those required by
Robbin and Salomon in the one-dimensional scenario, \cite[A1-A3]{RobSal:SFM}. On
the other hand, we let $D_2 : \Ga_c^\infty(M,F) \to L^2(M,F)$ be a first order
formally selfadjoint elliptic differential operator with bounded propagation
speed; Here $F$ is a hermitian vector bundle over $M$.

\begin{theorem}\label{I:ThmB}
Suppose that the conditions outlined before are satisfied. The
Dirac--Schr\"odinger operator $\la D_1(x) + i D_2 : \sD(D_1(x)) \cap \sD(D_2)
\to L^2(M,H \ot F)$ is an unbounded Fredholm operator for $\la > 0$ large
enough. Furthermore, its Fredholm index coincides with the integer given by
the interior $KK$--product $[D_1(\cdot)] \hot_{C_0(M)} [D_2] \in KK(\cc,\cc)$
under the canonical identification $KK(\cc,\cc) \cong \zz$.
\end{theorem}
Theorem \ref{I:ThmB}  is proved in the final Section \ref{s:GE}.
The proof of Theorem \ref{I:ThmA} consists of two main steps. First of
all one needs to prove that the pair $(D_1 \ti_{\Na} D_2, (X \hot_B Y)^2)$ is
an unbounded Kasparov module. This includes the selfadjointness and
\emph{regularity} of the unbounded product operator $D_1 \ti_{\Na} D_2$. Once
this is place we can apply the work of Kucerovsky \cite{Kuc:KKP} which
establishes general criteria for recognizing an unbounded Kasparov module as a
representative of the interior Kasparov product. The selfadjointness and
regularity of the unbounded product operator is handled by the following
result which the authors proved in a predecessor of this paper.

\begin{theorem}[{\cite[Theorem 7.10]{KaaLes:LGP}}]\label{t:KaaLes1}
Let $S$ and $T$ be two selfadjoint and regular operators on a Hilbert
$C^*$--module $E$. Suppose that there exists a core $\sE$ for $T$ such that
the following two conditions are satisfied:
\begin{enumerate}
\item We have the inclusions $(S - i\cd\mu)^{-1}(\xi) \in \sD(S) \cap \sD(T)$ and
  $T(S - i\cd\mu)^{-1}(\xi) \in \sD(S)$ for all $\mu\in\Rstar$ and all $\xi \in \sE$.
\item The unbounded operator $[S,T](S - i\cd\mu)^{-1} : \sE \to E$ extends to a
  bounded operator on $E$ for all $\mu\in\Rstar$.
\end{enumerate}
Then the unbounded anti--diagonal operator
\[
D = \begin{pmatrix}
0 & S - i\, T \\
S + i\, T & 0
\end{pmatrix}
: \bigl( \sD(S) \cap \sD(T) \bigr)^2 \to E^2
\]
is selfadjoint and regular. Here the power refers to the cartesian product
(\textit{i.e.,} direct sum) of modules. In particular, the inclusions
of domains $\sD(D)\hookrightarrow \sD(S), \sD(D)\hookrightarrow \sD(T)$
are continuous.
\end{theorem}

Finally, we briefly explain how the paper is organized:

In Section \ref{s:OStarAlg} we give a quick summary of the theory of operator
spaces and introduce the notion of an operator $*$-algebra. As a geometric
example we endow the $*$-algebra of $C^1$-functions vanishing at infinity
on a Riemannian manifold with an operator $*$-algebra structure.

Next, in Section \ref{s:OStarM} we discuss operator $*$-modules. As hinted at
earlier an operator $*$-module is a direct summand in the standard module over
an operator $*$-algebra. We think of operator $*$-modules as the analogues of 
Hilbert $C^*$-modules where the $C^*$-algebra is replaced by the more flexible
notion of an operator $*$-algebra.

Section \ref{s:COM} is devoted to the theory of connections on operator
$*$-modules. In analogy to the geometric theory of connections the space of
connections is an affine space modelled on a certain space of completely
bounded $A$-linear operators; there is a canonical connection, called the
Gra{\ss}mann connection, which arises from the operator $*$-module structure.

Section \ref{s:selfdirsch} is concerned with proving selfadjointness and
regularity of the unbounded product operator $D_1 \ti_{\Na} D_2$.
In the following Section \ref{s:IPU} we prove that the pair consisting of the
unbounded product operator and the interior tensor product of $X$ and $Y$ is
an unbounded Kasparov module. In particular we show that the resolvent of the
unbounded product operator is compact.

Kucerovsky's criterion \cite[Theorem 13]{Kuc:KKP}, stated in detail
as Theorem \ref{t:kuceodd} below, is then applied in Section \ref{s:URI} to ultimately prove
that our unbounded product construction yields an unbounded version of the
interior Kasparov product.

The final Section \ref{s:GE} then treats the geometric example of Dirac-Schr\"odinger
operators and proves Theorem \ref{I:ThmB}.

\section*{Acknowledgements} The authors would like to thank Alexander
Gorokhovsky for pushing one of us by asking the right questions
during an early stage of the project. 
Bram Mesland's work \cite{Mes:UBK} on unbounded $KK$-theory has been
constantly serving us as a source of inspiration and therefore we would
like to give him here a warm thank you as well. 

\section{Operator $*$--algebras}\label{s:OStarAlg} 
The purpose of this section is to introduce the notion of an \emph{operator
$*$--algebra}. This concept will be used throughout the paper. 
We start by reviewing the relevant theory of operator spaces. This subject is
well-treated in the literature. See for example Ruan \cite{Rua:SCA}, 
Blecher \cite{Ble:GHM} and the references therein.

\subsection{Preliminaries on operator spaces}\label{ss:POS}
Let $X$ be a Banach space over the complex numbers. The norm on $X$ will be
denoted by $\| \cd \|_X : X \to [0,\infty)$.

We will use the notation $M(\cc)$ for the $*$--algebra of infinite matrices
over $\cc$ with only finitely many entries different from $0$. $M(\cc)$
can be identified with the direct limit $\lim_{n\to \infty} M_n(\cc)$ 
where the limit is taken with respect to the inclusions
\begin{equation}\label{eq:inclu}
i_n : M_n(\cc) \to M_{n+1}(\cc), \qquad i_n(v) = \pmat{v & 0 \\
                                                         0 & 0}.
\end{equation}
We will refer to the $\cc$--algebra $M(\cc)$ as the \emph{finite matrices} over
$\cc$. 

For each $n \in \nn$ the $*$--algebra $M_n(\cc)$ is faithfully represented on
the Hilbert space $\cc^n$. We can thus endow $M_n(\cc)$ with the operator norm
coming from this faithful representation. The inclusions in \Eqref{eq:inclu}
then become isometries and we obtain a well-defined norm on the direct limit
$M(\cc)$. The norm on $M(\cc)$ will be denoted by $\|\cd\|_{\cc} : M(\cc) \to
[0,\infty)$. Alternatively, $M(\C)$ is faithfully represented 
on the Hilbert space $\ell^2(\Z_+)$ of square-summable sequences
and $\|\cd\|_{\C}$ is the norm induced by this representation. 

For a Banach space $X$ we denote by $M(X) := M(\cc) \ot_\cc X$ the
algebraic tensor product of the finite matrices over
$\cc$ and the Banach space $X$ . This vector space
has the structure of a $M(\cc)$--$M(\cc)$ bimodule in the obvious way. We will
refer to the bimodule $M(X)$ as the \emph{finite matrices} over $X$.

\begin{dfn}[Operator Space]\label{def:OSpace} A Banach space $(X,\|\cd\|)$
is called an \emph{operator space} if there exists a norm
$\|\cd\|_X : M(X) \to [0,\infty)$ (a priori to be distinguished from the norm,
$\|\cd\|$, on $X$!) on the finite matrices over $X$ such that
\begin{enumerate}
\item For any pair of finite matrices over $\cc$, $v,w \in M(\cc)$, and any
  finite matrix over $X$, $x \in M(X)$, we have the inequality
\[
\|v \cd x \cd w\|_X \leq \|v\|_\C \cd \|x\|_X \cd \|w\|_\C.
\]
\item For any pair of projections, $p,q \in M(\cc)$, with $pq = 0$ and any
  finite matrices $x,y \in M(X)$ we have the identity
\[
\|p x p + q y q \|_X = \max\{\|pxp\|_X, \|q y q\|_X\}.
\]
\item For any projection, $p \in M(\cc)$, of rank one and any element $x \in X$
  we have the identity $\|p \ot x\|_X = \|x\|$.
\end{enumerate}
\end{dfn}
The last condition ensures that the norm $\|\cd\|_X$ is compatible with the
given norm on $X$ and hence in the sequel we will always write $\|\cd\|_X$.

A closed subspace of a $C^*$--algebra is naturally an operator space.
Conversely, every operator space is isometric to a subspace of the algebra
$\B(\cH)$ of bounded operators on some Hilbert space, see~\cite{Rua:SCA}.
We will now review some standard constructions for operator spaces.

For $m \in \N$ we can make the $(m \times m)$--matrices over the operator
space $X$ into an operator space as follows: we define the norm on the finite
matrices over $M_m(X)$ using an appropriate identification 
$M_n\big(M_m(X) \big) \cong M_{nm}(X)$ of vector spaces for each $n \in \N$. 
Furthermore, we let $\Mbar{X}$ denote the completion of $M(X)$ in the operator
norm.  This normed space can be given the structure of an operator space by
using the identification $M_n(\Mbar{X}) \cong \Mbar{M_n(X)}$.

The direct sum $X^m = \op_{i=1}^m X$ can be embedded into the
$(m \times m)$--matrices over $X$ using the injective linear map
\begin{equation}\label{eq:1108045}  
\ph : X^m \to M_m(X), \qquad 
\ph\big( \{x_i\} \big) = \sum_{i=1}^m e_{i1} \ot x_i
                       =\pmat{ x_1 & 0 & \ldots & 0 \\
                               \vdots &  & \vdots & \\
                               x_m & 0 &\ldots  & 0}.
\end{equation}
Here $e_{i1} \in M_m(X)$ denotes the matrix with $1$ in position $(i,1)$ and
zeros elsewhere. This embedding gives $X^m$ the structure of an operator
space. Finally, the infinite direct sum, $X^\infty$, is defined as the
completion of the finite sequences $c_0(X) := \op_{i=1}^\infty X$ with respect
to the norm of the matrix algebra $M(X)$, \ie the closure of $c_0(X)$ inside
$\Mbar{X}$.  The operator space structure on $X^\infty$ is given by the
identification $M_m(X^\infty) \cong M_m(X)^\infty$. 

We will say that a continuous linear map $\ga : X \to Y$ between
the operator spaces $X$, $Y$ is 
\emph{completely bounded} if the supremum of operator norms,
$\sup_{n \in \nn} \| M_n(\al) \|, $ is finite. 
Here the notation $M_n(\al)$ stands for the continuous linear map
$\id\otimes \ga: M_n(X)\cong M_n(\C)\otimes X\to  M_n(\C)\otimes Y\cong M_n(Y)$ 
between the matrix spaces which is induced by $\ga$.
The vector space of completely bounded linear maps from
$X$ to $Y$ will be denoted by $CB(X,Y)$. 
This vector space becomes a Banach space when equipped with the norm 
defined by $\|\al\|_{\cb} := \sup_{n \in \nn} \|M_n(\al)\|$. 
In fact, it can be proved that the Banach space of
completely bounded maps can be turned into an operator space. The norms on
finite matrices come from the identification of vector spaces
\[
M_n\big( CB(X,Y) \big) \cong CB\big(X,M_n(Y)\big), \qquad\text{for } n\in \N;
\]
see \cite[p.~140]{EffRua:ROB}. We remark that a map 
$\al : X \to Y$ is completely bounded if and only if it
induces a bounded map $\al : \Mbar{X} \to \Mbar{Y}$; the latter is then
automatically completely bounded. 

Finally, we will say that the operator spaces $X$ and $Y$ 
are \emph{completely isomorphic} if there exists a completely bounded vector 
space isomorphism $\al : X \to Y$ with completely bounded inverse.
Note that the complete boundedness of the inverse is not automatic here. It would
follow if we knew that the induced map $\ga:\Mbar{X}\to \Mbar{Y}$ between
Banach spaces were bijective. This, however, does of course not follow
from the bijectivity and complete boundedness of $\ga:X\to Y$.

\subsection{Operator $*$--algebras}\label{ss:OStarA}

\begin{dfn}\label{def:realopalg}
Let $X$ be an operator space which is at the same time an algebra
over the complex numbers. We will call $X$ an \emph{operator algebra}
if the multiplication $m : X \times X \to X$ is completely bounded. 
This means that there exists a constant $K> 0$ such that
$\| x \cd y \|_X \leq K \cd \|x\|_X \cd \|y\|_X$ for all $x,y \in M(X)$.
\end{dfn}

We remark that a closed sub-algebra of a $C^*$--algebra is an operator
algebra. Indeed, the norm on the finite matrices is induced by the unique
$C^*$--norm on the matrices over the $C^*$--algebra. The converse is also true
by a theorem of D. P. Blecher \cite[Theorem 2.2]{Ble:CBC}, \cf also
Christensen--Sinclair \cite{ChrSin:RCB}, and Paulsen--Smith \cite{PauSmi:MMT}.

We are now ready to introduce the concept of an operator $*$--algebra, \cf
\cite[Def.~3.2.3]{Mes:UBK} and \cite[Def.~3.3]{Iva:CBC}.
\begin{dfn}\label{def:OpStarAlg}
Suppose that $X$ is an operator algebra. We will say that $X$ is an
\emph{operator $*$--algebra} if $X$ has a completely bounded involution 
$\da: X \to X$. Here, the involution on matrices is defined as usual by
transposing the matrix and replacing each entry $x$ by $x^\da$, \ie
$\{x_{ij}\}^\da = \{x_{ji}{}^\da\}$.
\end{dfn}

\begin{example} \label{ex:OpStarAlg}
An important example is provided by a closed subalgebra $A \su B$ of a
$C^*$--algebra $B$ together with a $*$--automorphism $\si : B \to B$ with
square equal to the identity and with $\si(x)^* \in A$ for all $x \in A$.
Defining a new involution on $A$ by $x^\da := \si(x)^*$ turns $A$ together
with $\da$ into an operator $*$--algebra. This
involution is actually completely isometric in the sense that 
$\|x^\da\| = \|x\|$ for all $x \in M(A)$.
Note that $A$ is not necessarily (neither with $*$ nor with $\da$) a 
sub--$*$--algebra of $B$.
\end{example}

For later reference we introduce the concept of
$\si$--unitality for operator $*$--algebras.

\begin{dfn}\label{def:SUnitality}
An operator $*$--algebra $X$ is called \emph{$\si$--unital} if there
exists a \emph{bounded} sequence $\{u_m\}$ in $X$ such that
\[
\lim_{m \to \infty} \|u_m \cd x - x\|_X = 0 = \lim_{m \to \infty} \|x \cd u_m
- x\|_X, \qquad \text{for all } x\in X.
\]
We will refer to the sequence $\{u_m\}$ as a (bounded) approximate unit for
$X$.
\end{dfn}
The boundedness of $\{u_m\}$ and the complete boundedness of the map
$\da$ ensure that with $\{u_m\}$ the sequences $\{u_m^\da\}$ and $\{u_m
u_m^\da\}$ are bounded approximate units as well.


\subsection{Geometric examples of operator $*$--algebras}\label{ss:GEOSA}
We will start by discussing another structure which, by using Example
\ref{ex:OpStarAlg}, can be used to construct an operator $*$--algebra.

\begin{prop}\label{p:OpStarAlg}
Assume that we are given
\begin{enumerate}
\item A Hilbert $C^*$--module $E$ over the $C^*$--algebra $B$.
\item A $*$--algebra, $\sA$, together with a $*$--homomorphism 
 $\pi:\sA \to \B(E)$ into the algebra of adjointable operators, $\B(E)$, on $E$.
\item A derivation $\de : \sA \to \B(E)$ which vanishes on $\ker \pi$ and
 satisfies $\de(a^*) = U \de(a)^* U$ for $a\in\sA$; here, $U$ is some unitary
 $U\in \B(E)$ which \emph{commutes} with the elements of $\sA$.  The
 derivation property means that $\de(a \cd b) = \de(a) \cd \pi(b) + \pi(a) \cd
 \de(b)$ for $a,b\in\sA$.
\end{enumerate}
Let $A_1 \su \B(E)$ denote the completion of $\pi(\sA)$ in the norm $\|\cd
\|_1 : \pi(a) \mapsto \|\pi(a)\|_\infty + \|\de(a)\|_\infty$; here $\|\cd\|_\infty$
denotes the operator norm on $\B(E)$.

Then the  sub--$*$--algebra $A_1 \su \B(E)$ can be given the structure of an
operator $*$--algebra by embedding it into $\B(E\oplus E)$ as follows:
\[
\varrho : a \mapsto 
\matr{cc}{
a & 0 \\
\de(a) & a
} \in \B(E \op E).
\]
\end{prop}
\begin{remark}\label{r:CompleteInclusion} 
 1. Since $\delta$ vanishes on $\ker \pi$ it descends to a derivation
 on $\pi(\sA)$ and, by continuity, to a derivation $A_1\to \B(E)$
 which is again denoted by $\delta$. By the very
 construction $\delta:A_1\to\B(E)$ is completely bounded. 

2. Let $A$ be the $C^*$--completion of $\pi(\sA)$,\ie the completion of
$\pi(\sA)$ with respect to the norm of $\B(E)$. Then the natural inclusion
$A_1\hookrightarrow A$ is \emph{completely} bounded.
Indeed, for $a\in M_n(A_1)$ and $\xi\in E^n$ we have
\[
\| a \cdot \xi \| \le \left\| \matr{cc}{
                                  a & 0 \\
                             \de(a) & a}
                             \matr{c}{ \xi \\ 0 }\right\|
          \le \|a\|_1\cd \|\xi\|.
\]

3. An important application of Proposition \ref{p:OpStarAlg} is the following,
\cf also the beginning of Sec.~\ref{s:COM} below. 
Consider a triple $(\sA, E, D)$, where $\sA$ and $E$ are as in (1) and (2) of 
Proposition \plref{p:OpStarAlg} and where $D$ is a selfadjoint densely defined 
unbounded operator in $E$ such that for each $a\in \sA$ the operator $\pi(a)$ 
maps the domain of $D$ into itself and the commutator $[D,\pi(a)]$ is in $\B(E)$. 

Put $\delta:\sA\to \B(E), a\mapsto [D,\pi(a)]\in\B(E)$. Then 
$\delta(a^*)=-[D,a]^*= i [D,a]^* i$, thus (3) of Proposition
\plref{p:OpStarAlg} is satisfied with $U=i\cd I$. As a consequence
we obtain a new triple $(A_1,E,D)$ satisfying (1) and (2) of
Proposition \ref{p:OpStarAlg} where now
$A_1\su\B(E)$ is an operator $*$--algebra, the inclusion
$A_1\hookrightarrow A$ into its $C^*$--completion is completely bounded
and $\delta=[D,\cd]:A_1\to \B(E)$ is completely bounded.
\end{remark}
\begin{proof}
The derivation property of $\delta$ ensures that $\varrho$ is an algebra
homomorphism; it is injective since $\sA$ is embedded into $\B(E)$.
Furthermore, the norm  induced by $\B(E\oplus E)$ on $\sA_1$ is equivalent to
$\|\cd\|_1$.

$\rho$, however, does not preserve the involution. Instead, we have
\begin{equation}\label{eq:1107231}  
 \varrho(a^*)=  V \varrho(a)^* V,\quad\text{with } 
 V= \pmat{  0 & U^*\\ U & 0}.
\end{equation}
$V$ is a unitary with $V^2=I$. Thus with the inner automorphism of 
$\B(E\oplus E)$ defined by $\sigma (\xi) = V \xi V$ we have $\sigma^2=I$ and 
$\varrho(a^*)=\sigma(\varrho(a))^*=:\varrho(a)^\da$. 
Thus according to Example \ref{ex:OpStarAlg} the algebra $\varrho(A_1)$ with involution
$\varrho(a)^\da:=\sigma(\varrho(a))^*$ is an operator $*$--algebra and
$\varrho$ is a $*$--isomorphism from $(\sA_1,*)$ onto $(\varrho(\sA_1),\da)$.
\end{proof}

Remark that Proposition \ref{p:OpStarAlg} is very much related to 
the example appearing in \cite[Sec.~3.1]{Mes:UBK}.

Next we apply Proposition \ref{p:OpStarAlg} to the algebra of
$C^1$--functions which vanish at infinity
on an oriented Riemannian manifold $M^m$ of dimension $m$. 
The orientation assumption is made for convenience only 
to have the Hodge $\star$ operator\footnote{At this point
we have to deal with at least three mathematical objects whose standard notation
is $*$: the involution, $\da$, on an operator $*$--algebra $\sA_1\subset\B(E)$,
which is to be distinguished from the native involution, $\ast$, on
the $C^*$--algebra $\B(E)$ and finally the Hodge star operator, $\star$. 
To distinguish the three objects notationally, we denote them by $\da$, $\ast$, and  $\star$,
respectively.} at our disposal without having to deal with the orientation
line bundle. With more notational effort the orientation assumption can
be disposed. We will use the notation $d$ for the
exterior derivative of complex valued forms on $M$. Furthermore, we let $\da$ denote
the involution on complex valued forms given by complex conjugation.

A section $s$ in a hermitian vector bundle, $\mathcal E$, over $M$
is said to \emph{vanish at infinity} if for each $\eps>0$ there exists
a compact subset $K\subset M$ such that $\|s(x)\|_{\mathcal E_x}<\eps$ for
all $x\in M\setminus K$. In short we write $\lim_{x\to \infty} s(x)=0$.
We denote by  $C^1_0(M)$ space of continuously differentiable complex valued
functions on $M$ for which
\begin{equation}\label{eq:1107232}  
 \lim_{x\to\infty} f(x)=0 \quad\text{and}\quad \lim_{x\to\infty}df(x)=0.
\end{equation}
Here, $d$ is the exterior derivative. $C^1_0(M)$ is a $*$--algebra with
involution defined by $f^\dag(x)=\ovl{f(x)}$.
Denote by
\begin{equation}\label{eq:1107233}  
    \Omega^p(M):=\Gamma^\infty(M,\gL^pT^*M\otimes \C)
\end{equation}
the smooth complex valued $p$--forms, \ie the smooth sections of the
hermitian vector bundle $\gL^p T^*M\otimes \C$. Since this vector bundle
is the complexification of a real vector bundle, complex conjugation is
well-defined for forms and  for $\go\in\gO^p(M)$ we therefore put
$\go^\da(x):=\ovl{\go(x)}$. The scalar product on the bundles
$T^*M\otimes\C$ is induced by the Riemannian metric. The metric
and the Hodge star operator are tied together by the formula
\begin{equation}\label{eq:1107238}  
 \inn{\go,\eta}\, \vol_x=\go\,\wedge\, \star\ovl{\eta}
\end{equation}
for $\go,\eta\in\gL^*T_x^*M\otimes\C$; here $\vol_x\in\gL^mT_x^*M$ denotes
the Riemannian volume element.

The space of bounded \emph{continuous} sections
$\Gamma_\textrm{b}(\gL^*T^*M\otimes \C)$ forms a graded commutative algebra which
acts by left multiplication on the space $L^2(\gL^*T^*M\otimes\C)$ of 
square-integrable sections of $\gL^*T^*M\otimes\C$. This representation is
faithful and hence we view $\Gamma_\textrm{b}(\gL^*T^*M\otimes \C)$
as a sub--$*$--algebra of the $C^*$--algebra $\B\bigl(L^2(\gL^*T^*M\otimes\C)\bigr)$.
Note that $C^1_0(M)$ is a sub--$*$--algebra of $\Gamma_\textrm{b}(\gL^*T^*M\otimes \C)
\subset  \B\bigl(L^2(\gL^*T^*M\otimes\C)\bigr)$.

The exterior derivative on functions now induces a natural derivation
\begin{multline}\label{eq:1107234}  
   \delta:C^1_0(M)\longrightarrow \B\bigl(L^2(\gL^*T^*M\otimes\C)\bigr), \quad
   \delta f (\go):= df\wedge \go,\\ \text{for } \go \in L^2(\gL^*T^*M\otimes\C).
\end{multline}
In order to apply Proposition \plref{p:OpStarAlg} to this situation we need
to find the unitary $U$ which commutes with $C^1_0(M)$, has square identity,
and $\delta f^\da=U (\delta f)^* U$. $U$ is provided by the Hodge $\star$
operator as follows:

Let 
\begin{multline}\label{eq:1107237}  
 \gl_0:= \sqrt{-1}^{m(m-1)/2}, \quad\text{and}\quad \gl_p:=
 (-1)^{p(p-1)/2}\gl_0,\\
 \text{for } p=0,\ldots,m.
\end{multline}
Then one checks the following identities:
\begin{align}
     \gl_p\cd\gl_{m-p} & = (-1)^{p(m-p)},     &&\text{for }
     p=0,\ldots,m,\label{eq:110723-a} \\
             \gl_{p+1} & = (-1)^p\,\gl_p,       &&\text{for } p=0,\ldots, m-1,\\
             \gl_{p+q}\cd\gl_{m-q} & = (-1)^{p(p-1)/2+q(m-p-q)}, &&\text{for }
             p,q, p+q\in\bigl\{1,\ldots,m\bigr\}.\label{eq:110723-c} 
\end{align}
where \Eqref{eq:110723-a} is a special case of \Eqref{eq:110723-c}.
Define the modified $\star$ operator on $p$--forms
by $\tstar_p:=\tstar_{|\gL^p}:=\gl_p\, \star_p$.
Then $\tstar$ is a unitary which commutes with the action of $C^1_0(M)$ 
and which satisfies $\tstar\,{}^2=I$.
Moreover, for any $p$--form, $\go\in\gO^p(M)$, the adjoint of the operator
$\ext(\go):=\go\wedge\cdot$ of exterior multiplication by $\go$
is given by 
\begin{equation}\label{eq:1108041}  
 \ext(\go)^*=(-1)^{p(p-1)/2}\, \tstar \ext(\go) \tstar,
\end{equation}
in particular
\begin{equation}\label{eq:1107236}  
    \bigl(\delta f\bigr)^* = \tstar \delta(f^\dag) \tstar.
\end{equation}
In view of Proposition \ref{p:OpStarAlg} we have proved.

\begin{prop}\label{conestar}
The map
\[
\pi : C^1_0(M) \to \B\big(L^2(\gL^*T^*M) \op L^2(\gL^*T^*M)\big), \qquad
\pi(f) = \pmat{f & 0 \\ df & f }
\]
is a $*$--isomorphism from $(C^1_0(M),\da)$ onto an operator $*$--subalgebra
of $\B\big(L^2(\gL^*T^*M) \op L^2(\gL^*T^*M)\big)$ with involution
given by $\pi(f)^\dag=V\pi(f^\dag)V$, where $V=\pmat{0&\tstar \\\tstar &0}$ and
$\star$ is defined on $\gL^p$ as $\gl_p\,\star_p$ with $\gl_p$ from
\Eqref{eq:1107237}.

This gives $C^1_0(M)$ naturally the structure of an
operator $*$--algebra.
\end{prop}

\begin{note}[Approximate unit on complete manifolds]\label{n:AppUniCom}
For future reference we note that if the oriented Riemannian manifold $M$ is
complete then we can choose a sequence of smooth compactly supported
functions $\{\chi_k\}$ such that
\begin{enumerate}
\item The image of each $\chi_k$ is contained in the interval $[0,1]$, \ie
 \hbox{$0\le \chi_k\le 1$}.
\item The exterior differentials, $d\chi_k$, converge to $0$ uniformly, more precisely
 it can be arranged that $\|d \chi_k\|_\infty \le 1/k$ for all $k\in\N$.
\item For each compact subset $K\subset M$ there is an index $k_0$ such
 that $\chi_k(x)=1$ for all $x\in K$ and all $k\ge k_0$.
 \end{enumerate}
 See, \eg \cite[Sec.~5]{Wol:ESA}, \cite[p.~117]{LawMic:SG}, or \cite[Lemma
3.2.4]{Les:OFT}. In particular, the sequence $\{\chi_k\}$ is a
bounded approximate unit for the operator $*$--algebra $C^1_0(M)$. Thus, in
the case of an oriented complete manifold the operator $*$--algebra $C^1_0(M)$
is $\si$--unital.
\end{note}

\section{Operator $*$--modules}\label{s:OStarM} 
The purpose of this section is to introduce the notion of an \emph{operator
$*$--module}. Stated a little vaguely, an
operator $*$--module is a direct summand in a standard module over an operator
$*$--algebra. In particular, by Kasparov's stabilization theorem each
countably generated Hilbert $C^*$--module is an operator $*$--module
\cite{Kas:HCM}, \cite[Sec.~13.6.2]{Bla:KTO2Ed}, \cite[Sec.~6.2]{Lan:HCM}.
However, the concept is more general. For example, we show that the Hilbert
space-valued $C^1$--functions which vanish at infinity on an oriented
Riemannian manifold form an operator $*$--module. We expect that, under
reasonable assumptions, $C^1$--sections of Hilbert bundles which vanish at
infinity are operator $*$--modules as well, \cf Remark \ref{r:HilbertBundles}
below. 

We start by giving the main operator algebraic definitions.
\begin{dfn}\label{def:OpMod}
Let $A$ be an operator algebra in the sense of Definition
\ref{def:realopalg}. Furthermore, let $X$ be a right-module over $A$.

We will then say that $X$ is a \emph{right operator module} over $A$
if $X$ is equipped with the structure of an operator space such
that the right action $X \times A \to X$ is completely bounded. This means that
there exists a constant $K>0$ such that 
\[
\| \xi \cd a\|_X \leq K \cd \|\xi\|_X \cd \|a\|_A, \qquad \text{for all }
 \xi \in M(X) \text{ and all } a \in M(A).
\]
\end{dfn}

Operator modules are well-treated in the
literature, see the survey \cite{ChrSin:SCB} and the references therein.

\begin{dfn}\label{def:HOpMod}
Suppose that $A$ is an operator $*$--algebra in the sense of 
Definition \ref{def:OpStarAlg} with involution $\da$
and let $X$ be a right operator module over $A$
in the sense of Definition \ref{def:OpMod}. 

We will say that $X$ is a \emph{hermitian operator module} if
there exists a \emph{completely bounded} pairing 
$\inn{\cd , \cd}_X : X \times X \to A$ satisfying the conditions
\begin{equation}\label{eq:sesqlin}
\begin{split}
\inn{\xi,\eta \cd \la + \rho \cd \mu}
& = \inn{\xi,\eta} \cd \la + \inn{\xi,\rho} \cd \mu \\
\inn{\xi,\eta \cd x}
& = \inn{\xi,\eta}\cd x \\
\inn{\xi,\eta} & = \inn{\eta,\xi}^\da
\end{split}
\end{equation}
for all $\xi,\eta,\rho \in X$, $x \in A$ and $\la,\mu \in \cc$.
\end{dfn}
The condition of complete boundedness means that the induced 
pairing of matrices
\begin{equation}\label{eq:CompBoundMatrixAlg}
\inn{\cd,\cd}_X : M(X) \times M(X) \longrightarrow M(A),
\qquad \inn{\xi,\eta}_{ij} = \sum_{k = 1}^\infty \inn{\xi_{ki},\eta_{kj}}
\end{equation}
is bounded in the sense that there exists a constant $K>0$ such that
\begin{equation}\label{eq:compinner}
\| \inn{\xi,\eta}\|_A \leq K \cd \|\xi\|_X \cd \|\eta\|_X,
 \qquad \text{for all } \xi,\eta \in M(X).
\end{equation}

Our first example of a hermitian operator module is the standard module over
$A$.

\begin{dfn}\label{def:StandardModule}
By the \emph{standard module} over the operator $*$--algebra $A$ we will
understand the completion of the finite sequences $c_0(A) \su M(A)$ in the
norm of $M(A)$, thus the closure of $c_0(A)$ in $\Mbar{A}$.
\end{dfn}

The standard module $A^\infty$ is an operator module over $A$. Furthermore,
we can define the pairing
\begin{equation}\label{eq:Pairing}
\inn{\cd,\cd} : A^\infty \ti A^\infty \to A\qquad
\inn{\{a_i\},\{b_i\}} = \sum_i a_i^\da
\cd b_i.
\end{equation}
To see that the sum is convergent we note that for $\{a_i\}_i, \{b_i\}_i\in A^\infty$
by definition the matrices, \cf \Eqref{eq:1108045},
\begin{equation}
\xi:=\pmat{ a_1 & 0 & \dots \\
       a_2 & 0 & \dots \\
    \vdots & \ddots & }, 
\quad
\eta:=\pmat{ b_1 & 0 & \dots \\
       b_2 & 0 & \dots \\
    \vdots & \ddots & }
\end{equation}
are in $\Mbar{A}$ and the pairing \Eqref{eq:CompBoundMatrixAlg} yields
\begin{equation}\label{eq:3}  
 \inn{\xi,\eta}_{A^\infty}= \xi^\da \cd \eta = 
       \pmat{ \sum_i a_i^\da \cd b_i & 0 & \dots \\
                    0 & \dots \\
		    \vdots & }.  
\end{equation}
The properties of the pairing \Eqref{eq:CompBoundMatrixAlg}, in particular
\Eqref{eq:compinner}, then imply not only the convergence of the rhs of
\Eqref{eq:Pairing} but also that the pairing $\inn{\cd,\cd}$
in \Eqref{eq:Pairing} is completely bounded. 
$A^\infty$ is thus a hermitian operator module.

We are now ready for the main definition of this section.
\begin{dfn} \label{def:OpStarMod}
Suppose that $X$ is a hermitian operator module over the operator $*$--algebra $A$. 
We will say that $X$ is an \emph{operator $*$--module} if it is completely isomorphic 
to a direct summand in the standard module $A^\infty$. To be more precise, there exist a
completely bounded selfadjoint idempotent $P : A^\infty \to A^\infty$ and a completely
bounded isomorphism of hermitian operator modules $\al : X \to P
A^\infty$. Here the selfadjointness of $P$ means that $\inn{P\{a_i\},\{b_i\}}
= \inn{\{a_i\},P\{b_i\}}$ for all sequences $\{a_i\}\, , \, \{b_i\} \in A^\infty$.
\end{dfn}

Suppose that $X = PA^\infty$ and $Y = Q A^\infty$ are operator
$*$--modules. A finite sequence $\xi=\{\xi^k\}_{1\le k\le N}\in c_0(X)$ may be
thought of as a matrix $\{\xi^k_l\}_{k,l=1}^\infty$ with entries in $A$ where
only finitely many rows contain nonzero entries,\ie
\begin{equation}\label{eq:1108141}  
        \xi=\pmat{\xi_1^1 & \xi_2^1 &\dots   \\
                    \vdots &        & \vdots  \\
                    \xi_1^N &  \xi_2^N &\dots \\
                    0       & 0        &\dots       \\
                    \vdots  & \text{\huge\bf $0$} & \dots }.
\end{equation}
Here, each row $\{\xi^k_l\}_{l\in\N}$ lies in $A^\infty$. The reader should
be warned that the matrix $\xi$ does not necessarily lie in $\Mbar{A}$.
Rather, the transposed
matrix, $\xi^t$,\ie the infinite matrix with columns given by 
$\xi^1,\xi^2,\ldots\in X\su A^\infty$ is in $\Mbar{A}$. 
The infinite matrix $(\xi^t)^{\da}\in\Mbar{A}$ is obtained from $\xi^t \in \Mbar{A}$ 
by applying the completely bounded involution $\da : \Mbar{A} \to \Mbar{A}$.
Equivalently, $(\xi^t)^\da$ is obtained by replacing each entry $\xi^k_l$ of
the matrix $\xi$ by $(\xi^k_l)^\da$.

For each pair of finite sequences $\xi \in c_0(X) \T{ and } \eta
\in c_0(Y)$ we let $\te_{\xi,\eta} : Y \to X$ denote the completely bounded
module map defined by
\[
\te_{\xi,\eta}(\rho) := \xi^t \cd \inn{\eta^t,\rho} = \sum_{k=1}^\infty \xi^k
\cd \inn{\eta^k,\rho}, \qquad \text{for } \rho \in Y,
\]
in fact
\begin{equation}\label{eq:ThetaCBEst}  
    \|\te_{\xi,\eta} \|_\cb \le C\cd \|\xi\|_X \cd \|\eta \|_Y
\end{equation}
with some constant $C>0$ independent of $\xi,\eta$.
Note that $\theta_{\xi,\eta}:Y\to X$ is given by matrix multiplication
by the infinite matrix $\xi^t\cd (\eta^t)^\da\in \Mbar{A}$.

\begin{prop}\label{p:approxcompact}
Suppose that $X$ is an operator $*$--module over the $\si$--unital
operator $*$--algebra $A$. Then there exists a sequence
$\{w^m\}_{m=1}^\infty$ of elements in $c_0(X)$ such that $\te_{w^m,w^m}(\rho)
\to \rho$ for all $\rho \in X$. Furthermore, the sequence can be chosen to be
bounded in the sense that $\sup_{m \in \nn} \|(w^m)^t\|_X < \infty$.
This means that $\sup_{m\in \N}\|\theta_{w^m,w^m}\|_\cb<\infty$.
\end{prop}
\begin{proof} It suffices to prove the claim for $X = A^\infty$:
 for if $\{w^m\}_{m=1}^\infty$ is such a sequence for $A^\infty$ then
 $\{Pw^m\}_{m=1}^\infty$ does the job for $X=PA^\infty$.

Let $\{u^m\}_{m=1}^\infty$ be an approximate unit for $A$ in the sense
of Definition \ref{def:SUnitality}. For each $m \in \N$ we let 
$(v^m)^t = (e_1 u_m,\ldots,e_m u_m)$. Here $e_i
u_m \in A^\infty$ denotes the sequence with $u_m$ in position $i$ and zeros
elsewhere. Then 
\[
\|(v^m)^t\|_{A^\infty} 
= \|(e_1 u_m,\ldots,e_m u_m)\|_{A^\infty}
= \|1_m \ot u_m \|_A
= \|u_m\|.
\]
Here $1_m\in M(\C)$ denotes the $(m \ti m)$ unit matrix viewed as an
idempotent in $M(\C)$. Furthermore, we have used item (3) of 
Definition \plref{def:OSpace}. This proves that 
$\sup_{m \in \nn} \|(v^m)^t\| < \infty$ since
the approximate unit $\{u_m\}$ is bounded in $A$.
Therefore, to prove that the sequence $\{\te_{v^m,v^m}\}$ converges strongly to the
identity we only need to show that $\te_{v^m,v^m}(\al) \to \al$ for each
finite sequence $\al = \sum_{i=1}^k e_i a_i \in c_0(A)$. For each $m \geq k$
we have
\begin{multline*}
\| \te_{v^m,v^m}(\al)  - \al \|_{A^\infty}
= \|v^m \inn{v^m,\al} - \al\|_{A^\infty}
= \|\sum_{i=1}^m e_i u_m \inn{e_i u_m,\al} - \al \|_{A^\infty} \\
= \| \sum_{i=1}^k e_i \cd (u_m u_m^\da a_i - a_i)\|_{A^\infty}
\leq \sum_{i=1}^k \| u_m u_m^\da \cd a_i - a_i\|_A.
\end{multline*}
As remarked after Definition \ref{def:SUnitality} the sequence
$\{u_m u_m^\da \}$ is a bounded approximate unit for $A$ as well,
hence the right hand side converges to $0$ and the claim about strong
convergence follows. The last claim is a consequence of \Eqref{eq:ThetaCBEst}.
\end{proof}

We expect that the theory of operator $*$--modules fits nicely
into Blecher's theory of rigged modules \cite[Def.~3.1]{Ble:GHM}. 
Obvious candidates for the structure maps are induced by the sequence 
$\{w^m\}$ of elements in $c_0(X)$ using the module
structure as well as the completely bounded pairing. 

Furthermore, each countably generated Hilbert $C^*$--module is an operator
$*$--module. This can be seen as a consequence of Kasparov's stabilization
theorem \cite{Kas:HCM}, \cite[Sec.~13.6.2]{Bla:KTO2Ed}, \cite[Sec.~6.2]{Lan:HCM}.

\subsection{The standard module for $C^1_0(M)$} \label{ss:SMCM}
We end this section by computing the standard module for the algebra
$C^1_0(M)$, $M$ an oriented Riemannian manifold, \cf Subsection \ref{ss:GEOSA}.
For a separable Hilbert space $H$ we denote by $C_0^1(M,H)$ the space
of $H$-valued continuously differentiable maps $f:M\to H$ satisfying
\Eqref{eq:1107232}. 
$C^1_0(M,H)$ becomes a Banach space when equipped with the norm 
\begin{equation}\label{eq:1108042}  
\|f\|_1 = \sup_{x \in M}|\inn{f(x),f(x)}|^{1/2} + \sup_{x \in M} |\inn{df(x),df(x)}|^{1/2}.
\end{equation}
Furthermore, pointwise multiplication gives $C^1_0(M,H)$ the structure of
a module over $C^1_0(M)$.

\begin{prop}\label{conemodu}
Let $M$ be an oriented Riemannian manifold and let $H$ be a separable Hilbert
space. The standard module, $\bigl(C^1_0(M)\bigr)^\infty$,
over the operator $*$--algebra $A = C^1_0(M)$ is then isomorphic to
$C^1_0(M,H)$. 
\end{prop}
\begin{proof}
Let $\{e_i\}$ be an orthonormal basis for $H$. Recall that the
submodule $\op_{i=1}^\infty C_0^1(M)$ of finite sequences is dense in the
standard module $\bigl(C^1_0(M)\bigr)^\infty$. Likewise, the linear span,
$\lspan_{\cc}\big\{e_i f \, | \, i \in \nn \, , \, f \in C_0^1(M) \big\}$, is
a dense submodule of $C^1_0(M,H)$. These two submodules are isometric by the
isomorphism $\{f_i\} \mapsto \sum_{i=1}^\infty e_i f_i$. 
Indeed, the norm of multiplication by the matrices
\begin{equation}\label{eq:1108043}  
\pmat{ f_1 & 0 &\ldots & 0\\
       \vdots &  &\vdots & \\
       f_k & 0 & \ldots  & 0}, \qquad
\pmat{ df_1 & 0 &\ldots & 0\\
       \vdots &  &\vdots & \\
       df_k & 0 & \ldots  & 0}
\end{equation}
on $L^2(\gL^* T^*M\ot \C)^k$ is easily seen to be
$\sup_{x\in M} \big(\sum_{i=1}^k |f_i(x)|^2\big)^{1/2}$,
resp. $\sup_{x \in M} \big(\sum_{i=1}^k \inn{df_i(x),df_i(x)} \big)^{1/2}$.
Consequently,
\begin{equation}\label{eq:1108044}  
\|f\|_1  = \|\{f_i\}\|_\infty + \|\{df_i\}\|_\infty,
\end{equation}
where the notations $\|\{f_i\}\|_\infty$ and $\|\{df_i\}\|_\infty$ are
shorthand for the operator norms of the multiplication by the matrices
in \eqref{eq:1108043}. Since the natural norm on
$\bigl(C^1_0(M)\bigr)^\infty$ is given by the right hand side of 
\Eqref{eq:1108044} we reach the conclusion.
\end{proof}

\begin{remark}\label{r:HilbertBundles}
With the above result in mind it seems to be a worthwhile task to characterize
the operator $*$--modules over $C^1_0(M)$,  thus the ``completely bounded''
direct summands in the module $C^1_0(M,H)$, \eg in the case of a complete
oriented manifold. We expect that many interesting Hilbert bundles will appear
in this way. We hope to explore this in a subsequent publication.
\end{remark}

\section{Connections on operator $*$--modules}\label{s:COM}   

In order to ease reference to it we are going to introduce some
standard notation in the form of a numbered proclaim.

\begin{conv}\label{StaConA}
Let $X_1 = P\StM{A_1}$ be an operator $*$--module over an operator $*$--algebra
$A_1$. We assume that $A_1 \su A$ sits as a dense $*$--subalgebra inside a
$C^*$--algebra $A$ and that the inclusion $i : A_1 \hookrightarrow A$ is completely
bounded, \cf Proposition \ref{p:OpStarAlg} and Remark \ref{r:CompleteInclusion}. 
The operator $*$--algebra norm on $A_1$ will be denoted by
$\|\cd\|_1$ and the $C^*$--algebra norm on $A$ will be denoted by
$\|\cd\|$. 
\end{conv}

Given such an $X_1$ the completely bounded selfadjoint
idempotent $P : \StM{A_1} \to \StM{A_1}$ extends to an orthogonal projection $P :
\StM{A} \to \StM{A}$. Indeed, 
\[
\|P\{a_i\}\|^2 
= \| \inn{\{a_i\},P\{a_i\}} \|
\leq \|\{a_i\}\| \cd \|P\{a_i\}\|
\]
for all sequences $\{a_i\} \in \StM{A_1}$. We let $X = P\StM{A}$ denote the Hilbert
$C^*$-module over $A$ defined by $P : \StM{A} \to \StM{A}$. The inclusion $X_1 \to X$
is then completely bounded and compatible with both the inner products and the
module actions.

\begin{conv}\label{StaConB} Let $A_1$ be an operator $*$-algebra
as in Convention \ref{StaConA}.
Let $(A,Y,D)$ be a triple consisting of:
\begin{enumerate}
\item An $A$--$B$ Hilbert $C^*$--bimodule $Y$. That is, $Y$ is a Hilbert
$C^*$--right module over the $C^*$--algebra $B$\footnote{we will for brevity
also use the term ``Hilbert $B$--module'' instead of the somewhat lengthy
``Hilbert $C^*$--module over $B$''.} together with a $*$-representation
$\pi:A\to \B_B(Y)$. 
\item A selfadjoint densely defined unbounded operator $D : \sD(D) \to Y$ in $Y$
such that 
\begin{enumerate}
 \item\label{StaNot2a} Each $a\in A_1$ maps the domain of $D$ into itself, 
 \item\label{StaNot2b} the commutator with $D$ yields a completely bounded map 
       $[D,\cd] : A_1 \to \B(Y)$ on the operator $*$--algebra $A_1 \su A$. 
\end{enumerate}       
\end{enumerate}
\end{conv}
\cf also \cite[Def.~6]{Kuc:KKP}.  If $B=\C$ these are, up to the requirement of 
\emph{complete} continuity and a missing compactness assumption, 
the axioms for a spectral triple $(Y,A_1,D)$,
\cf \cite[Def.~3.1 and Remark 3.2]{Hig:RIT}. For unbounded operators one has
to be careful with domains; the condition \eqref{StaNot2a}, which is
very important, is often slightly obscured in the literature.
The assumption of complete boundedness in \eqref{StaNot2b} is not very 
restrictive, \cf Remark \ref{r:CompleteInclusion}, 3. 

We will mostly suppress $\pi$ from the notation and write
$a\cdot y$ for the action, $\pi(a)$,  of $a\in A$ on $y\in Y$. The
$C^*$--algebra of bounded adjointable operators on $Y$ is denoted by 
$\B_B(Y)$; if no confusion is possible we will also omit the subscript $B$.

In the sequel we will repeatedly use the interior tensor product 
``$\hot_A$'' of $C^*$--modules, see \cite[Prop.~4.5]{Lan:HCM} and
\cite[Sec.~13.5]{Bla:KTO2Ed}. Recall that 
the Hilbert $B$--module $X\hot_A Y$ is the completion of the algebraic
tensor product $X\ot_A Y$ with respect to the inner product
\begin{equation*}
     \inn{x_1\otimes y_1, x_2\otimes y_2}_B = \inn{y_1, \inn{x_1,x_2}_A y_2}_B,
     \qquad \text{for } x_1,x_2\in X, y_1, y_2\in Y.
\end{equation*}
The $C^*$--algebra $\B_B(Y)$ is, as every $C^*$--algebra, a Hilbert module
over itself. Thus $X\hot_A \B_B(Y)$ is a Hilbert $B$--module. It is also
an $A$--right module via the representation $\pi:A\to \B_B(Y)$. The action
of $\B_B(Y)$ on $Y$ gives rise to the contraction map
\begin{equation}    
    c:(X \hot_A \B(Y)) \ot Y \longrightarrow X \hot_A Y, \quad
      x\ot T\ot y\mapsto x\ot Ty \label{eq:1108052}  
\end{equation}
and the inner product on $X$ induces the pairing
\begin{equation}
    X \times X\hot_A\B_B(Y)\longrightarrow \B_B(Y),\quad ( x, y \ot T) = \inn{x,y} \cd T
                 \label{eq:1108051}   
\end{equation}                 

After these preparations we are ready to introduce the main concept
of this section. 
\begin{dfn}\label{def:HermConn}
With the notation of Conventions \ref{StaConA} and \ref{StaConB}
introduced before we call a completely bounded 
linear map $\Na_{D} : X_1 \to X \hot_A
\B(Y)$ a \emph{$D$--connection} if
\[
\Na_{D}(x \cd a) = \Na_{D}(x) \cd a + x \ot [D,a],
\qquad \text{for all } x \in X_1, a \in A_1.
\]
A $D$--connection is called \emph{hermitian} if additionally
\[
[D,\inn{x_1,x_2}] = \big( x_1,\Na_{D}(x_2) \big) - \big( x_2,
\Na_{D}(x_1) \big)^*, \qquad \text{for all } x_1,x_2 \in X_1.
\]
Here $(\cdot,\cdot)$ denotes the pairing introduced in \Eqref{eq:1108051}.

Furthermore, we write $c(\Na_{D})$ for the composition of maps
\[
\begin{CD}
c(\Na_{D}) : X_1 \ot Y @>(\Na_{D} \ot 1)>> (X \hot_A \B(Y)) \ot Y @>c>> X
\hot_A Y.
\end{CD}
\]
\end{dfn}

Before passing to the examples of hermitian $D$--connections we 
record.
\begin{lemma}\label{l:innercomp} Assume Conventions \ref{StaConA}, \ref{StaConB}. 
Suppose that $\Na_{D} : X_1 \to X \hot_A \B(Y)$ is a hermitian
$D$--connection. Then we have the identity
\[
\binn{c(\Na_{D})(x_1 \ot y_1),x_2 \ot y_2}\\
= \binn{x_1 \ot y_1, c(\Na_{D})(x_2 \ot y_2)}
- \binn{y_1, [D,\inn{x_1,x_2}](y_2)},
\]
between inner products on $X \hot_A Y$ and inner products on $Y$, for all
$x_1,x_2 \in X_1$ and all $y_1,y_2 \in Y$.
\end{lemma}
\begin{proof}
The result follows from the computation
\[
\begin{split}
 \binn{&c(\Na_{D})(x_1 \ot y_1),x_2 \ot y_2} 
= \binn{\Na_{D}(x_1)(y_1),x_2 \ot y_2} \\
& = \binn{\big(x_2,\Na_{D}(x_1)\big)(y_1),y_2}  
   = \binn{y_1, \big(x_2,\Na_{D}(x_1)\big)^*(y_2)} \\
& = \binn{y_1, \big(x_1,\Na_{D}(x_2) \big)(y_2)} -
    \binn{y_1,[D,\inn{x_1,x_2}](y_2)} \\
&  = \binn{x_1 \ot y_1, c(\Na_{D})(x_2 \ot y_2)}
   - \binn{y_1, [D,\inn{x_1,x_2}](y_2)}.
\end{split}
\]
Here we have in fact only used the hermitian property of $\Na_{D}$.
\end{proof}

\subsection{The Gra{\ss}mann connection}\label{levicivi}
We continue assuming Convention \ref{StaConA}.
We shall now see that our operator $*$--module $X_1 = P\StM{A_1}$ carries a
canonical hermitian $D$--connection provided that an extra mild condition is
satisfied. This connection is basically the restriction of the commutator
$[D,\cd]$ to the operator $*$--module. We will thus term this connection the
Gra{\ss}mann $D$--connection.

In order to explain our extra condition we introduce the
$C^*$--algebra of $D$--$1$--forms, \cf \cite[Chap.~IV.1]{Con:NG}.

\begin{dfn}\label{def:DForms} Let $(X_1,A_1)$ and $(A,Y,D)$ as in 
Conventions \ref{StaConA} and \ref{StaConB}. Let $\gO^1_{D}\su \B_B(Y)$ 
be the smallest $C^*$--subalgebra of $\B_B(Y)$
such that $[D,a_1] \T{ and } \pi(a) \in \Om^1_{D}$ for all $a_1 \in A_1$ 
and all $a \in A$. We endow $\Om^1_{D}$ with its natural structure of an
$A$--$\Om^1_{D}$ Hilbert--$C^*$--module. Let $\pi:A\to \B(\gO^1_D)$ be
the action of $A$ by left multiplication on $\gO^1_D$.

The action $\pi : A \to \Om^1_{D}$ is called
\emph{essential} if $\pi(A)\gO^1_D$ is dense in $\gO^1_D$;
then also $\pi(A_1)\gO^1_D$ is dense in $\gO^1_D$.
\end{dfn}

From now on we assume that the action $\pi : A \to \Om^1_{D}$ is
\emph{essential}. In particular, we have an isomorphism $A^\infty \hot_A
\Om^1_{D} \cong (\Om^1_{D})^\infty$ of Hilbert $C^*$--modules.

We are now ready to define the Gra{\ss}mann $D$--connection.

\begin{dfn}[Gra{\ss}mann $D$--connection]\label{def:LCConn}
By the \emph{Gra{\ss}mann $D$--connection} on $X_1$ we will understand the
completely bounded linear map $\Gc_{D}$ obtained as the composition
\[
X_1 \longhookrightarrow A_1^\infty \xrightarrow{[D,\cd]}
(\Om^1_{D})^\infty \cong A^\infty \hot_A \Om^1_{D} 
\xrightarrow{P \ot 1}  X \hot_A \Om^1_{D} \longrightarrow X \hot_A \B(Y).
\]
Here the second map, $[D,\cd]$, is given by applying the completely bounded
derivation $[D,\cd]$ to each entry in a sequence.
\end{dfn}

\begin{prop}\label{p:LCHerm}
The Gra{\ss}mann $D$--connection is a hermitian $D$--connection in the
sense of Definition \ref{def:HermConn}.
\end{prop}
\begin{proof}
We start by proving that $\Gc_{D}$ is a $D$--connection. Thus,
let $x = \{a_i\} \in X$ and let $a \in A_1$. Then
\[
\begin{split}
\Gc_{D}(x \cd a) & = (P \ot 1)\{ [D,a_i \cd a] \}  \\
    & = (P \ot 1)\{[D,a_i]\} \cd a + (P \ot 1)( \{a_i\} \ot [D,a] ) \\
    & = \Gc_{D}(x) \cd a + x \ot [D,a].
\end{split}
\]
Here we have suppressed various maps in order to obtain a cleaner
computation. To show that $\Gc_{D}$ is hermitian we compute
for $\{a_i\}\, , \, \{b_i\} \in X$
\[
\begin{split}
\big[D, \inn{\{a_i\},\{b_i\}}\big]
& = \sum_{i=1}^\infty [D,a_i^* \cd b_i]
= \sum_{i=1}^\infty a_i^* [D,b_i] - \big( \sum_{i=1}^\infty b_i^*
[D,a_i]\big)^*
\\
& = \big(\{a_i\},\{[D,b_i]\})  - \big(\{b_i\},\{[D,a_i]\})^* \\
& = \big( \{a_i\}, \Gc_{D}\{b_i\} \big)
- \big( \{b_i\}, \Gc_{D}\{a_i\} \big)^*,
\end{split}
\]
which is the desired identity.
\end{proof}

\subsection{Comparison of connections}
We are now interested in comparing different hermitian $D$--connections. We
shall apply a general result on automatic boundedness for operator
$*$--modules. 

\begin{prop}\label{p:autobound} Let $(X_1,A_1)$ be as in Convention \ref{StaConA}
with $A_1$ being $\sigma$--unital. Then any completely bounded $A_1$--linear
map $\ga:X_1 \to Z$ into an operator module $Z$ over $A$ 
extends to a completely bounded $A$--linear
map, $\al : X \to Z$.  
\[ \xymatrix{X_1=P\StM{A_1} \ar[d]\ar[r]^-\ga &Z  \\ X =P\StM{A}
\ar@{-->}[ur]_\ga  &} \] 
\end{prop}
\begin{proof}
Let $\{\te_{w^m,w^m}\}$ denote a sequence of completely bounded operators as
in Proposition \ref{p:approxcompact}. Then we have for any finite matrix $x \in M_n(X_1)$ over $X_1$
\[
\begin{split}
\al(x) & = \lim_{m \to \infty} \al(\te_{w^m,w^m}(x))
= \lim_{m \to \infty} \al\big((1_n\ot (w^m)^t ) \cd \binn{1_n \ot (w^m)^t 1_n,x}
\big) \\
& = \lim_{m \to \infty} \al(1_n \ot (w^m)^t) \cd \binn{1_n\ot (w^m)^t,x}.
\end{split}
\]
Here $(w^m)^t \ot 1_n$ refers to the $(n \ti n)$ diagonal matrix which has the
row $(w^m)^t$ on the diagonal. It follows that
\[
\begin{split}
\|\al(x)\|_Z 
& \leq C_1 \cd \limsup_{m \to \infty} \|\al\|_{\cb} \cd \|(w^m)^t\|_{X_1} \cd
\|\binn{(w^m)^t \ot 1_n,x}\|_A \\
& \leq C_2 \cd \limsup_{m \to \infty} \|\al\|_{\cb} \cd \|(w^m)^t\|_{X_1}^2 \cd
\|x\|_X,
\end{split}
\]
with constants $C_1,C_2 > 0$ independent of the size
$n$ of the matrix. But this proves the claim since $\sup_{m \in \nn}
\|(w^m)^t\| < \infty$ is finite.
\end{proof}

\begin{prop}\label{p:compself}
Let $(X_1,A_1)$ be as in Convention \ref{StaConA} with $A_1$
being $\sigma$--unital. Furthermore let $(A,Y,D)$ as in Convention \ref{StaConB}.
Consider two $D$--connections $\Na_{D}, \wit\Na_{D} : X_1 \to X \hot_A \B(Y)$. 
Then the difference $\Na_{D} - \wit \Na_{D}$ extends to a completely bounded
$A$--linear operator $\Na_{D} -  \wit \Na_{D} : X \to X \hot_A \B(Y)$.

If additionally the two connections are hermitian then 
the completely bounded operator $c( \Na_{D} - \wit\Na_{D}) : X \hot_A Y \to X
\hot_A Y$ is selfadjoint. Here, $c$ is the contraction map
defined in \Eqref {eq:1108052}.
\end{prop}
\begin{proof}
The first claim is a consequence of Proposition \ref{p:autobound} since the connection
property implies that the difference $\Na_{D} - \wit \Na_{D}$ is
$A_1$--linear.

Hence we get a completely bounded operator 
$c(\Na_{D} - \wit \Na_{D}) : X \hot_A Y \to X \hot_A Y$.

To prove the last claim 
let us fix some elements $x_1,x_2 \in X_1$ and $y_1,y_2 \in Y$. By Lemma
\ref{l:innercomp} we can calculate as follows
\[
\begin{split}
& \binn{ c( \Na_{D} - \wit\Na_{D})(x_1 \ot y_1),x_2 \ot y_2} \\
& \q = \binn{x_1 \ot y_1, c(\Na_{D})(x_2 \ot y_2)}
- \binn{y_1, [D,\inn{x_1,x_2}](y_2)} \\
& \qq - \binn{x_1 \ot y_1, c(\wit\Na_{D})(x_2 \ot y_2)}
+ \binn{y_1, [D,\inn{x_1,x_2}](y_2)} \\
& \q = \binn{x_1 \ot y_1, c(\Na_{D} - \wit\Na_{D})(x_2 \ot y_2)}.
\end{split}
\]
But this computation proves the proposition.
\end{proof}

\section{Selfadjointness and regularity of the unbounded product operator}\label{s:selfdirsch}
Let $(X_1=P\StM{A_1},A_1)$  be as in Convention \ref{StaConA}
with $A_1$ being $\si$--unital; as usual $X = P \StM{A}$ denotes the associated Hilbert
$C^*$--module. Assume furthermore, that $D_1 : \sD(D_1) \to X$ 
is an unbounded selfadjoint and regular operator on $X$.

Secondly, we assume that we are given a triple $(A,Y,D_2)$
as in Convention \ref{StaConB} and
assume that the action $\pi : A \to \B(Y)$ is essential. 
Furthermore, we will assume that the action $A \to \B(\Om^1_{D_2})$ is essential
as well, \cf Def. \ref{def:DForms}.

We assume that these assumptions are in effect for the remainder of this
section. The aim of this section is to prove, under an additional commutator
hypothesis, the selfadjointness and regularity of the
\emph{unbounded product operator} defined by
\begin{equation}\label{eq:UnbProdOp}  
\begin{split}
D_1 \ti_\Na D_2  := &\pmat{
0 & D_1 \ot 1 - i\, 1 \ot_{\Na} D_2 \\
D_1 \ot 1  + i\, 1 \ot_{\Na} D_2 & 0
} \\ 
& :\bigl( \sD(D_1 \ot 1) \cap \sD(1 \ot_{\Na} D_2) \bigr)^2  \to (X \hot_A Y)^2.
\end{split}
\end{equation}
Here $\Na : X_1 \to X \hot_A \B(Y)$ is \emph{any} hermitian
$D_2$--connection. 


\subsection{Selfadjointness and regularity of $1 \ot_{\Na} D_2$}\label{selfregright}
We start with a discussion of the right leg $1 \ot_{\Na} D_2$ of the unbounded
product operator for which no additional assumptions are needed.

The first thing to do is to investigate the situation where the hermitian
$D_2$--connection is the Gra{\ss}mann $D_2$--connection constructed in
Subsection \ref{levicivi}. The case of a general hermitian $D_2$--connection
will then follow from the comparison result in Proposition \ref{p:compself}.

We let $\diag(D_2) : \sD(\diag(D_2)) \to Y^\infty$ denote the
selfadjoint and regular diagonal operator induced by $D_2 : \sD(D_2) \to
Y$. Now, since the action $\pi : A \to \B(Y)$ is assumed to be essential we
have an isomorphism $A^\infty \hot_A Y \cong Y^\infty$ of Hilbert
$C^*$--modules. In particular, we can make sense of the projection $P \ot 1 :
Y^\infty \to Y^{\infty}$. We define the unbounded operator $1
\ot_{\Gc} D_2$ as the composition
\[
\begin{CD}
1 \ot_{\Gc} D_2 \, :\,
\sD(1 \ot_{\Gc} D_2) @>>> \sD(\diag(D_2)) 
@>\diag(D_2)>> Y^\infty
@>(P \ot 1)>> X \hot_A Y.
\end{CD}
\]
Here the domain is given by the intersection $\sD(1 \ot_{\Gc} D_2) :=
\sD(\diag(D_2)) \cap (X \hot_A Y)$.

\begin{lemma}\label{expform}
For each $x \in X_1$ and each $y \in \sD(D_2)$ we have the explicit formula
\[
(1 \ot_{\Gc} D_2 )(x \ot y) 
= x \ot D_2(y) + c(\Gc_{D_2})(x \ot y).
\]
In particular, the unbounded operator $1 \ot_{\Gc} D_2$
is densely defined.
\end{lemma}
\begin{proof}
Let us fix elements $x = \{a_i\} \in X_1 = PA_1^\infty$ and $y \in
\sD(D_2)$. We then have that
\[
\begin{split}
& x \ot D_2(y) + c(\Gc_{D_2})(x \ot y)
= (P \ot 1)\{a_i \cd D_2(y)\} + (P \ot 1)\{[D_2,a_i](y)\} \\
& \q = (P \ot 1)\{D_2(a_i \cd y)\}
= (P \ot 1)\diag(D_2)\{a_i \cd y\}
= (1 \ot_{\Gc} D_2)(x \ot y).
\end{split}
\]
But this is the desired identity.
\end{proof}

Before we continue we remark that each element $T \in \Mbar{A_1}$
determines both a completely bounded operator $T : A_1^\infty \to A_1^\infty$
on the standard module and a bounded adjointable operator $T : Y^\infty \to
Y^\infty$. We will make use of this observation in the next lemmas.
%

\begin{lemma}\label{commcompact}
 Let $T \in \Mbar{A_1}$. Then the associated bounded adjointable operator $T
: Y^\infty \to Y^\infty$ preserves the domain of $\diag(D_2)$ and the
commutator $[\diag(D_2),T] = [D_2,T] : \sD(\diag(D_2)) \to Y^\infty$
extends to a bounded adjointable operator with the estimate $\|[D_2,T]\| \leq
\|T\|_{A_1} \cd \|[D_2,\cd]\|_{\cb}$ on the operator norm.
\end{lemma}
\begin{proof}
The statement follows by approximating $T$ with finite matrices and by using
the complete boundedness of the commutator $[D_2,\cd]$.
\end{proof}

The result of Lemma \ref{commcompact} allows us to analyze the relation
between the projection $P \ot 1$ and the diagonal operator $\diag(D_2)$. We
recall that our operator $*$--algebra $A_1$ is assumed to be $\si$--unital.

\begin{lemma}\label{commbounded}
The projection $P \ot 1  : Y^\infty \to Y^\infty$ preserves the domain of the
diagonal operator $\diag(D_2)$ and the commutator
\[
[P \ot 1,\diag(D_2)] : \sD(\diag(D_2)) \to Y^\infty
\]
extends to a bounded adjointable operator on the Hilbert $C^*$--module $Y^\infty$.
\end{lemma}
\begin{proof} By Proposition \ref{p:approxcompact} there exists
as sequence $\{w^m\}\su c_0(A_1^\infty)$ such that 
$\sup\|(w^m)^t\|_{A_1^\infty} < \infty$ and
$\te_{w^m,w^m}(\rho) \to \rho$ for all $\rho \in A_1^\infty$. 
Proposition \ref{p:approxcompact}. 
Put $\te_m = (Pw^m)^t ((w^m)^t)^* \in \Mbar{A_1}$, \cf the
paragraph before Proposition \ref{p:approxcompact}. Then, according to loc.~cit.,
$\te_m(y) \to (P \ot 1)(y)$ for all $y \in Y^\infty$ and
\[
\sup_{m \in
  \nn}\|\te_m\|_{\cb} \leq \sup_{m \in \nn} \|\te_m\|_{A_1} \leq C \cd \|P\|_{\cb}
\cd \sup_{m \in \nn} \|(w^m)^t\|^2_{A_1^\infty} < \infty.
\]
The statement of the lemma therefore
follows from Lemma \ref{commcompact} if we can prove that the sequence
$\{[\diag(D_2),\te_m](y)\}$ converges for each $y$ in a dense
subspace of $Y^\infty$, therefore it suffices to consider
vectors of the form $\al \ot z$ where $\al\in A_1^\infty$ and $z \in Y$. 
However, for elements of this kind we have
\[
[D_2,\te_m](\al \ot z) = [D_2, \te_m(\al)](z) - \te_m[D_2,\al](z),
\]
where the right hand side converges to $[D_2, P(\al)](z) - (P \ot
1)[D_2,\al](z)$.
\end{proof}

We are now ready to prove the main result of this section: The selfadjointness
and regularity of the unbounded operator $1 \ot_{\Gc} D_2$.

\begin{theorem}\label{t:extright}
Let $(X_1=P\StM{A_1},A_1)$  be as in Convention \ref{StaConA}
with $A_1$ being $\si$--unital and let the triple $(A,Y,D_2)$
satisfy the assumptions of Convention \ref{StaConB}; furthermore,
assume that the actions $\pi : A \to \B(Y)$ and $A \to \B(\Om^1_{D_2})$ 
are essential.

Then for any hermitian
$D_2$--connection $\Na_{D_2} : X_1 \to X \hot_A \B(Y)$ 
the unbounded operator 
\[
1 \ot_{\Na} D_2 := 1 \ot_{\Gc} D_2 + c(\Na_{D_2} -
\Gc_{D_2})
: (X \hot_A Y) \cap \sD(\diag(D_2)) \to X \hot_A Y
\]
is selfadjoint and regular. Furthermore we have the explicit formula
\[
(1 \ot_{\Na} D_2)(x \ot y)
= x \ot D_2(y) + c(\Na_{D_2})(x \ot y)
\]
whenever $x \in X_1$ and $y \in \sD(D_2)$.
\end{theorem}
\begin{proof}
Let us first note that the class of selfadjoint regular operators is stable under
bounded adjointable perturbations. This follows easily from a 
Neumann series argument and,\eg \cite[Lemma 9.8]{Lan:HCM}, \cf also \cite{Wor:UEA}. 
To express it differently, 
in a slightly overblown fashion, it also follows from the Kato--Rellich Theorem
for Hilbert $C^*$--modules \cite[Theorem 4.4]{KaaLes:LGP}.

We may assume that $\nabla$ is the Gra{\ss}mann connecton $\Gc$.
The general case then follows from Proposition \ref{p:compself} and
Lemma \ref{expform}. 

So let $Q = P \ot 1$. The diagonal operator $\diag(D_2)$ can be
written
\[
\begin{split}
\diag(D_2) & = Q \diag(D_2) Q + (1 - Q) \diag(D_2) (1 - Q) \\
& \qq + Q \diag(D_2) (1 - Q) + (1 - Q) \diag(D_2) Q \\
& = Q \diag(D_2) Q + (1 - Q) \diag(D_2) (1 - Q) \\
& \qq + [Q,\diag(D_2)](1 - Q) - (1-Q)[Q, \diag(D_2)].
\end{split}
\]
Now, from Lemma \ref{commbounded} we infer that
\[
[Q,\diag(D_2)](1 - Q) - (1-Q)[Q, \diag(D_2)] :
Y^\infty \to Y^\infty
\]
is selfadjoint and bounded. The operator
\begin{equation}\label{eq:diagcomp}
Q\diag(D_2)Q
+ (1 - Q) \diag(D_2) (1 - Q) : \diag(D_2) \to Y^\infty
\end{equation}
thus differs from the selfadjoint regular operator $\diag(D_2)$ 
by a bounded selfadjoint operator. As noted at the beginning of this
proof this implies the selfadjointness and regularity of the operator
\eqref{eq:diagcomp}. But since $1 \ot_{\Gc} D_2$ is just the 
compression $Q \diag(D_2) Q$ we reach the conclusion.
\end{proof}


\subsection{Selfadjointness and regularity of the product operator}
We are now in a position where we can prove selfadjointness and regularity
results for unbounded product operators of the form $D_1 \ti_{\Na} D_2$ where
$\Na_{D_2}$ is a hermitian $D_2$--connection. See the beginning of Section
\ref{s:selfdirsch}. Our main tool will be the Theorem \ref{t:KaaLes1} on
selfadjointness and regularity for sums of operators:

In our case the roles of $S$ and $T$ are played by $D_1 \ot 1$ and $1
\ot_{\Na} D_2$. The Hilbert $C^*$--module $E$ in loc.~cit.~is given 
by the interior tensor product $X \hot_A Y$. The selfadjointness and regularity of
$D_1$ is easily seen to imply the selfadjointness and regularity of
the unbounded operator $D_1 \ot 1 : \sD(D_1) \hot_A Y \to X \hot_A Y$.
The unbounded operator $1 \ot_{\Na} D_2$ is selfadjoint and regular for any hermitian
$D_2$--connection $\Na_{D_2}$ by Theorem \ref{t:extright}.

\begin{theorem}\label{t:regself} 
In the situation of Theorem \ref{t:extright} let in addition $D_1$ be
a selfadjoint regular operator on $X$.

Suppose that there exists a hermitian $D_2$--connection $\Na^0_{D_2} : X_1 \to
X \hot_A \B(Y)$ such that the conditions in Theorem \ref{t:KaaLes1} are
satisfied for $S := D_1 \ot 1$ and $T := 1 \ot_{\Na^0} D_2$. Then the
unbounded product operator
\[
\begin{split}
D_1 \ti_{\Na} D_2 & := \pmat{0 & D_1 \ot 1 - i\, 1 \ot_{\Na} D_2 \\
D_1 \ot 1 + i\, 1 \ot_{\Na} D_2 & 0} \\
& \q : \bigl(\sD(D_1 \ot 1) \cap \sD(1 \ot_{\Na} D_2)\bigr)^2 \to (X \hot_A Y)^2
\end{split}
\]
is selfadjoint and regular for \emph{any} hermitian $D_2$--connection
$\Na_{D_2} : X_1 \to X \hot_A \B(Y)$.
\end{theorem}
\begin{proof}
The selfadjointness and regularity of $D_1 \ti_{\Na^0} D_2$ is a consequence
of Theorem \ref{t:KaaLes1}. The statement for a general hermitian
$D_2$--connection follows since $D_1 \ti_{\Na^0} D_2 - D_1 \ti_{\Na} D_2$
extends to a selfadjoint bounded operator by Proposition \ref{p:compself}. Here
we use again the fact that the class of selfadjoint regular operators
is stable under selfadjoint bounded perturbations as already mentioned at the
beginning of the proof of Theorem \ref{t:extright}. 
\end{proof}
\mpar{clumsy, maybe make a Lemma}
\mpar{
(IT IS POSSIBLE TO IMPROVE THEOREM \ref{t:regself}. INDEED IT CAN BE PROVED THAT
$X_1 \ot_{A_1} \sD(D_2)$ IS A CORE FOR $1 \ot_{\Na} D_2$. I DON'T KNOW WHETHER
WE SHOULD INCLUDE THIS IMPROVEMENT. SEE THE OLD VERSION).}

\section{The interior product of unbounded Kasparov modules}\label{s:IPU}

Let $A$, $B$ and $C$ be three $C^*$--algebras. Let us start by recalling some
terminology from \cite{BaaJul:TBK}, \cf also \cite{Kuc:KKP}.

\begin{dfn}\label{def:UKM}
By an \emph{unbounded Kasparov $A$--$B$ module}  we will understand a pair
$(X,D)$ consisting of a countably generated $A$--$B$ Hilbert $C^*$--bimodule $X$ 
and an unbounded selfadjoint and regular operator $D : \sD(D) \to X$ on $X$ 
such that:
\begin{enumerate}
\item There is a dense $*$--subalgebra $\sA\su A$ such that each
 $a\in\sA$ maps $\dom(D)\to \dom(D)$ and 
 the commutator $[D,a] : \sD(D) \to X$ extends to a
  bounded operator.
\item The resolvent operator $a \cd (D - i)^{-1} \in \cK(X)$ is
  $B$--compact for all $a \in A$.
\end{enumerate}

We will say that $(X,D)$ is \emph{even} when we have a $\zz_2$--grading
operator $\gam \in \B(X)$ such that $\pi(a)\gam - \pi(a)\gam = 0$ for all
$a \in A$ and $D \gam + \gam D = 0$. 

An unbounded Kasparov $A$--$B$ module without a grading operator is referred
to as being \emph{odd}.
\end{dfn}

We remind the reader of Proposition \ref{p:OpStarAlg}  and Remark
\ref{r:CompleteInclusion}, 3. although in this section we will not make
use of the completion $A_1$ of $\sA$.

Let us fix an odd unbounded Kasparov $A$--$B$ module $(X,D_1)$ and an odd
unbounded Kasparov $B$--$C$ module $(Y,D_2)$. The aim of this section is to
find sufficient conditions for the existence of an ``unbounded product'' of
$(X,D_1)$ and $(Y,D_2)$. When it exists, the unbounded product will be an even
unbounded Kasparov $A$--$C$ module which depends on the choice of a connection
$\Na$ up to selfadjoint perturbations. The unbounded product operators will be
of the ``Dirac--Schr\"odinger--type'' $D_1 \ti_{\Na} D_2$ which we considered in
the previous section. Let us give some relevant definitions.

\begin{dfn}\label{def:EssKas}
$(Y,D_2)$ is called \emph{essential} if the action of $B$ on $Y$
is essential and the derivation $[D_2,\cd]$ is essential. 
That is, $B\cd Y$ is dense in $Y$ and $B \cd \Om^1_{D_2}$ is dense in
the $\Om^1_{D_2}$. Recall from \ref{def:DForms} that $\Om^1_{D_2} \su \B(Y)$ denotes the
smallest $C^*$--subalgebra such that $b \, , \, [D_2,b] \in \Om^1_{D_2}$ for
all $b$ in the dense $*$--subalgebra $\sB\su B$ according to Def.
\ref{def:UKM}. 
\end{dfn}

The next definition is related to the notion of a correspondence which appears
in the Ph.D.~thesis of B. Mesland \cite[Sec.~6]{Mes:UBK}. The
conditions which we require are however substantially weaker than those
advocated in loc.~cit.

\begin{dfn}\label{def:Correspondence}
Suppose that $(Y,D_2)$ is essential. By a \emph{correspondence} from $(X,D_1)$
to $(Y,D_2)$ we will understand a pair $(X_1,\Na^0)$ consisting of an operator
$*$--module $X_1$ over a $\si$--unital operator $*$--algebra $B_1$ and a
completely bounded hermitian $D_2$--connection $\Na^0 : X_1 \to X \hot_B
\B(Y)$ such that
\begin{enumerate}
\item The operator $*$--module $X_1 \su X$ is a dense subspace of $X$ and the
  operator $*$--algebra $B_1 \su B$ is a dense $*$--subalgebra of
  $B$. The inclusions are completely bounded and compatible with module
  structures and inner products.
\item Each $b\in\ B_1$ maps the domain of $D_2$ into itself and
 the derivation $[D_2,\cd] : B_1 \to \B(Y)$ is 
  completely bounded on $B_1$.
\item The commutator $[1 \ot_{\Na^0} D_2,a] : \sD(1 \ot_{\Na^0} D_2) \to X
  \hot_B Y$ is well--defined and extends to a bounded operator on $X \hot_B Y$
  for all $a \in \sA$.
 \item For any $\mu\in \R\setminus \{0\}$ the unbounded operator
\[
[D_1\ot 1, 1 \ot_{\Na^0} D_2](D_1 \ot 1 - i\cd\mu)^{-1} : \sD(1 \ot_{\Na^0}
D_2) \to X \hot_B Y
\]
is well--defined and extends to a bounded operator on $X \hot_B Y$.
\end{enumerate}
\end{dfn}

Here, (4) is an abbreviation for the properties (1) and (2) in 
Theorem \ref{t:KaaLes1} for $S=D_1\ot 1, T=1\ot_{\Na^0} D_2$.

Furthermore, we remark that the domain of $1 \ot_{\Na^0} D_2$ can be replaced by a core for
$1 \ot_{\Na^0} D_2$ in requirement (3) and (4) of Definition
\ref{def:Correspondence}.

\begin{dfn}
Suppose that $(X_1,\Na^0)$ is a correspondence from $(X,D_1)$ to
$(Y,D_2)$. Let $\Na_{D_2} : X_1 \to X \hot_B \B(Y)$ be a hermitian
$D_2$--connection. By the \emph{unbounded interior product} of $(X,D_1)$ and
$(Y,D_2)$ with respect to $\Na_{D_2}$ we will understand the pair $\big((X \hot_B
Y)^2, D_1 \ti_{\Na} D_2\big)$. Here $(X \hot_B Y)^2$ is a $\zz_2$--graded
$A$--$C$ Hilbert $C^*$--bimodule. The grading is given by the grading operator
$\gam := \diag(1,-1)$.
\end{dfn}

We shall see that the unbounded interior product is an unbounded even Kasparov
$A$--$C$ bimodule. We remark that the selfadjointness and regularity
condition was proved in Theorem \ref{t:regself}. Furthermore, the boundedness of
the commutator $[D_1 \ti_{\Na} D_2,a]$ for all $a\in \sA$ follows from the
third condition in Definition \ref{def:Correspondence}. The only real issue is therefore
compactness of the resolvent. This problem will occupy the rest of the
section. We remark that the unbounded interior product only depends on the
choice of connection up to selfadjoint perturbations. This is a consequence of
Proposition \ref{p:compself}.

We start with a small compactness result.

\begin{lemma}\label{l:diagcomp}
Suppose that $B$ is $\si$--unital and that the action $B \to Y$ is
essential. Let $K \in \cK(B^\infty,X)$ be a $B$--compact operator. Then for 
$z\in \C\setminus \R$ the bounded operator
\[
(K \ot 1) \big(\diag(D_2) - z \big)^{-1} \in \cK(B^\infty \hot_B Y, X
\hot_B Y)
\]
is $C$--compact. Here we have suppressed the isomorphism of Hilbert $C^*$--modules
$B^\infty \hot_B Y \cong Y^\infty$.
\end{lemma}
\begin{proof}
By the resolvent identity it suffices to prove the claim for $z=i$. 
Let $\{u_m\}$ denote the countable approximate unit for $B$. We then have a
countable approximate unit $\{\te_m\}$ for the compact operators on $B^\infty$
defined by $\te_m := \sum_{i=1}^m \te_{e_i \cd u_m, e_i \cd u_m} \in 
\cK(B^\infty)$. Here $e_i \cd u_m \in B^\infty$ is the vector in the standard
module with $u_m \in B$ in position $i \in \nn$ and zeros elsewhere. We
therefore only need to prove that the bounded operator $(\te_m \ot 1) \cd
\big(\diag(D_2) - i \big)^{-1} \in \B(Y^\infty)$ is $C$--compact for all $m
\in \nn$.

Let us fix some $m \in \nn$. Using the identification $Y^\infty \cong H \hot
Y$ where ``$\hot$'' denotes the exterior tensor product of Hilbert
$C^*$--modules we get the identity
\begin{equation}\label{eq:comid}
(\te_m \ot 1) \cd \big(\diag(D_2) - i \big)^{-1} 
= p_m \ot \Bigl(u_m u_m^* \cd (D_2 - i)^{-1}\Bigr) \in \B(H \hot Y).
\end{equation}
Here $p_m : H \to H$ denotes the finite rank orthogonal projection onto the
subspace $\T{span}_{\cc}\{e_i\}_{i=1}^m$ where $\{e_i\}$ is an orthonormal
basis for the separable Hilbert space $H$. Since both of the factors in the 
tensor product on the rhs of \eqref{eq:comid} are ($\C$-- resp. $C$--) compact 
we get that $(\te_m \ot 1) \cd \big(\diag(D_2) - i \big)^{-1} \in \cK(Y^\infty)$ 
is $C$--compact and the lemma is proved.
\end{proof}



\begin{prop}\label{komright}
Suppose that condition \textup{(1)} and \textup{(2)}
in Definition \ref{def:Correspondence} are
satisfied. Let $\Na_{D_2} : X_1 \to X \hot_B \B(Y)$ be a hermitian
$D_2$--connection and let $K \in \cK(X)$ be a $B$--compact operator. Then the
bounded adjointable operator 
$(K \ot 1) (1 \ot_{\Na} D_2 - z)^{-1} \in \cK(X \hot_B Y)$ is $C$--compact
for all $z\in\C\setminus \R$.
\end{prop}
\begin{proof}
We start by recalling that by Proposition \ref{p:compself}
the difference of unbounded operators $1 \ot_\Na D_2 - 1 \ot_{\Gc} D_2$ 
extends to a bounded selfadjoint operator. 
By the resolvent identity it is therefore
sufficient to prove the claim for the Gra{\ss}mann $D_2$--connection 
$\Gc : X_1 \to X \hot_B \B(Y)$ and $z=i$.

By assumption, \cf Convention \ref{StaConA} resp. 
Proposition \ref{p:autobound},
we can assume that $X_1 = PB_1^\infty$ and $X =
PB^\infty$ where $P : B_1^\infty \to B_1^\infty$ is a completely bounded
selfadjoint idempotent. 
Since $(Y,D_2)$ is assumed to be essential we have
$X\hot_B Y = PB^\infty\hot_B Y=(P\ot 1)(B^\infty\hot_B Y)$ and
$Y^\infty\cong B^\infty\hot_B Y$. Put $Q = P \ot 1 : Y^\infty\cong
B^\infty\hot_B Y \to X \hot_B Y$. Then
\begin{align}
 (K \ot 1)&(1 \ot_{\Gc} D_2 - i)^{-1} Q \nonumber\\
            = &(K \ot 1)(Q\diag(D_2)Q - iQ)^{-1} Q \nonumber\\
           = &(K P \ot 1)(Q\diag(D_2)Q + (1- Q)\diag(D_2)(1-Q)-i)^{-1}
               \label{eq:201110031}\\
          & :Y^\infty \to Q Y^\infty.\nonumber \end{align}
The difference $Q\diag(D_2) Q + (1 - Q)\diag(D_2)(1-Q) - \diag(D_2)$ is bounded 
and selfadjoint by Lemma \ref{commbounded}. Another application of the
resolvent equation then shows that the $C$--compactness of the rhs of 
\Eqref{eq:201110031} is equvialent to that of
\[
          (K P \ot 1)(\diag(D_2)  - i)^{-1} : B^\infty\hot_B Y \to X\hot_B Y.
\]
Since $KP : B^\infty \to X$ is $B$--compact the result follows from Lemma \ref{l:diagcomp}
\end{proof}

The next result will allow us to conclude compactness results for the
resolvent $(D_1 \ti_{\Na} D_2 - i)^{-1}$ by looking at operators of the form
$(K \ot 1)(1 \ot_{\Na} D_2 -i)^{-1}$.


We are now ready to prove the main result of this section.

\begin{theorem}\label{t:prodmod}
Suppose that $(X_1,\Na^0)$ is a correspondence between the unbounded odd
Kasparov modules $(X,D_1)$ and $(Y,D_2)$. Let $\Na_{D_2} : X_1 \to X \hot_B \B(Y)$ 
be any completely bounded hermitian $D_2$--connection. Then the pair
$(D_1 \ti_{\Na} D_2, X \hot_B Y)$ is an even unbounded Kasparov $A$--$C$
module which only depends on $\Na_{D_2}$ up to selfadjoint bounded
perturbations.
\end{theorem}
\begin{proof}
Let us fix some countable approximate unit $\{\te_m\}$ for the compact
operators on $X$.

As observed at the beginning of this section, Theorem \ref{t:regself}
applies to $D_1\ti_\Na D_2$. That means that Theorem \ref{t:KaaLes1}
applies to $S=D_1\ot 1$ and $T=1\ot_{\Na^0} D_2$. Since the difference 
$1 \ot_{\Na} D_2 - 1 \ot_{\Na^0} D_2$ is bounded
selfadjoint by Proposition \ref{p:compself} 
it follows in particular that we have the following commutative diagram of 
continuous inclusion maps of Hilbert $C^*$--modules:
\begin{equation}\label{eq:DomInc}
 \xymatrix{     &   \sD(1\ot_\Na D_2)^2 \ar[dr]^{\io_2}& \\
 \sD(D_1\ti_\Na D_2)^2 \ar[ur]^{\io_1}\ar[rr]^{\io}\ar[dr]^{\io_3} &   &  (X\hot_B Y)^2\\
                &   \sD(D_1 \ot 1)^2 \ar[ur]^{\io_4}& 
                }
\end{equation}
We need to prove that for $a\in A$ the bounded operator $\pi(a) \ci \io : \sD(D_1 \ti_\Na
D_2) \to (X \hot_B Y)^2$ is $C$--compact. 
Here, $\pi(a) \in \B(X \hot_B Y)^2$ is given by
the action of $A$ on the first component in the interior tensor product.

Let $m \in \nn$ and consider the composition 
$(\te_m \ot 1) \ci \pi(a) \ci \io = (\te_m \ot 1) \ci \pi(a) \ci \io_2 \ci \io_1
: \sD(D_1 \ti_\Na D_2) \to (X \hot_B Y)^2$, \cf \Eqref{eq:DomInc}.

The operator $(\te_m \ot 1) \ci \pi(a) \ci \io_2 : \sD(1
\ot_\Na D_2)^2 \to (X \hot_B Y)^2$ is $C$--compact by Proposition
\ref{komright}. Hence $(\te_m \ot 1) \ci \pi(a) \ci \io \in 
\cK\big( \sD(D_1 \ti_\Na D_2), (X \hot_B Y)^2 \big)$ is $C$--compact for all $m \in
\nn$.

On the other hand we have the identity
\[
(\te_m \ot 1) \ci \pi(a) \ci \io_4 \ci \io_3 = (\te_m \ot 1) \ci \pi(a) \ci \io
: \sD(D_1 \ti_\Na D_2) \to (X \hot_B Y)^2.
\]
But the sequence of operators $\{ (\te_m \ot 1) \ci
\pi(a) \ci \io_4\}$ in $\B\big(\sD(D_1 \ot 1)^2,(X \hot_B Y)^2\big)$
converges in operator norm to the bounded operator $\pi(a) \ci \io_4 \in 
\B\big(\sD(D_1 \ot 1)^2, (X \hot_B Y)^2\big)$. Indeed, this follows by noting
that $\pi(a) \ci (D_1 \ot 1 - i )^{-1}$ is of the form $K \ot 1$ where $K \in
\cK(X)$ is compact.

We have thus proved that $\pi(a) \ci \io \in \B\big( \sD(D_1 \ti_\Na D_2),
(X \hot_B Y)^2\big)$ is the limit in operator norm of a sequence of compact
operators. It is therefore compact and the theorem is proved.
\end{proof}
\section{Unbounded representatives for the interior Kasparov product}\label{s:URI} 
Let $A$, $B$ and $C$ be $C^*$--algebras where $A$ is separable and $B$ is
$\si$--unital. We then have the interior Kasparov product $\hot_B : KK^1(A,B)
\ti KK^1(B,C) \to KK^0(A,C)$ which is a bilinear and associative pairing of
abelian groups \cite{Kas:OKF}, \cite[Sec.~18]{Bla:KTO2Ed}. The purpose of this section is to show that
the unbounded interior product which we constructed in the last section is an
unbounded version of the interior Kasparov product.

Let $(X,D_1)$ and $(Y,D_2)$ be odd unbounded Kasparov modules for $(A,B)$ and
$(B,C)$ respectively. By the work of Baaj and Julg \cite{BaaJul:TBK}
the bounded transform $F : (X,D) \mapsto \big(X,D(1+ D^2)^{-1/2}\big)$ 
provides classes $F(X,D_1)\in KK^1(A,B)$ and $F(Y,D_2)\in KK^1(B,C)$.

\begin{dfn} We say that an even unbounded Kasparov
$A$--$C$ bimodule $(Z,D)$ \emph{represents} the interior Kasparov product of
$(X,D_1)$ and $(Y,D_2)$ if
\[ 
F(X,D_1) \hot_B F(Y,D_2) = F(Z,D)
\]
in the even $KK$--group $KK^0(A,C)$. 
\end{dfn}

We shall see that the existence of a correspondence $(X_1,\Na^0)$ from
$(X,D_1)$ to $(Y,D_2)$ implies that the even unbounded Kasparov module 
$((X\hot_B Y)^2,D_1\ti_{\Na} D_2)$ represents the interior Kasparov product for
any hermitian $D_2$--connection $\Na$. Our main tool will be a general result
which is an adaption of a theorem of D. Kucerovsky to the case of the interior
Kasparov product between two odd $KK$--theory groups. The result can thus be
proved by an application of D. Kucerovsky's theorem together with some
understanding of formal Bott--periodicity in $KK$--theory, see for example
\cite[Cor.~17.8.9]{Bla:KTO2Ed}.

For each $x \in X$ we will use the notation $T_x : Y^2 \to (X \hot_B Y)^2$ for
the multiplication operator $T_x : (y_1,y_2) \mapsto ((x \ot y_1),(x \ot
y_2))$; $T_x$ is bounded adjointable. Furthermore, we let
$\si_1,\si_2 \in M_2(\cc)$ denote the matrices
\[
\si_1 := \matr{cc}{0 & -i \\ i & 0} \quad \text{and} \quad \si_2 := \matr{cc}{0 & 1 \\ 1 & 0}.
\]

\begin{theorem}[{\cite[Theorem 13]{Kuc:KKP}}]\label{t:kuceodd}
Let $(\pi_1,X,D_1)$ and $(\pi_2,Y,D_2)$ be two odd unbounded Kasparov modules
for $(A,B)$ and $(B,C)$, respectively. Let $(\pi_1 \ot 1, (X \hot_B Y) \op
(X \hot_B Y),D)$ be an even unbounded Kasparov $A$--$C$ module, where $(X \hot_B
Y)^2$ is $\zz_2$--graded by the grading operator $\gam =
\diag(1,-1)$. Suppose that the following conditions are satified:
\begin{enumerate}
\item The commutator
\[
\left[ \matr{cc}{D & 0 \\ 0 & D_2 \cd \si_1} , \matr{cc}{0 & T_x \\ T_x^* & 0 }
\right] : \sD(D) \op \sD(D_2)^2 \to (X \hot_B Y)^2 \op Y^2
\]
is well--defined and extends to a bounded operator on $(X \hot_B Y)^2 \op Y^2$ for
all $x$ in a dense subset of $A \cd X$.
\item The domain of $D$ is contained in the domain of $(D_1 \ot 1)\cd \si_2$.
\item There exists a constant $C > 0$ such that 
\[
\binn{Dz,(D_1 \ot 1)\cd \si_2 (z)} + \binn{(D_1 \ot 1)\cd \si_2 (z),Dz} \geq
-C\inn{z,z}
\]
for all $z \in \sD(D)$.
\end{enumerate}
Then the even unbounded Kasparov module $(D,(X \hot_B Y)^2)$
represents the interior Kasparov product of $(D_1,X)$ and $(D_2,Y)$.
\end{theorem}

The second condition in the above theorem can be slightly weakened \cite[Lemma 10]{Kuc:KKP}. 
However, for our purposes the stronger requirement on the domains is sufficient.

Let us fix two odd unbounded Kasparov modules $(X,D_1)$ and $(Y,D_2)$ for
$(A,B)$ and $(B,C)$, respectively. Furthermore, we assume that we have 
a correspondence $(X_1,\Na^0)$ from $(X,D_1)$ to $(Y,D_2)$. 
In particular, $(Y,D_2)$ is essential in the sense of Definition \ref{def:EssKas} 
and the operator $*$--algebra $B_1$ is $\si$--unital.

In the next lemmas we will show that $(X,D_1)$, $(Y,D_2)$ and $((X \hot_B
Y)^2,D_1 \ti_{\Na^0} D_2)$ satisfy the conditions of Theorem \ref{t:kuceodd}. We
remark that $((X \hot_B Y)^2,D)$ is an even unbounded Kasparov $A$--$C$ module
by Theorem \ref{t:prodmod}.

We let $\sF := \pi_1(\sA) \cd (D_1 -i)^{-1}(X_1)$ and remark that $\sF$ is a
dense subset of $A \cd X$.

\begin{lemma}\label{domaininclus}
For each $x \in \sF$ we have that $T_x(\sD(D_2)) \su \sD(D_1 \ot 1) \cap \sD(1
\ot_{\Na^0} D_2)$ and $T_x^*(\sD(D_1 \ot 1) \cap \sD(1 \ot_{\Na^0} D_2)) \su
\sD(D_2)$.
\end{lemma}
\begin{proof}
Let $x := a \cd (D_1 - i)^{-1}(\xi) \in \sF$ with $\xi \in X_1$ and $a \in
\sA$.

Let $y \in \sD(D_2)$. We need to show that that $x \ot y \in \sD(D_1 \ot 1)
\cap \sD(1 \ot_{\Na^0} D_2)$. Since $\sF \su \sD(D_1)$ we get that $T_x(y) = x
\ot y \in \sD(D_1 \ot 1)$. Next, by definition of $x \in \sF$, we
find $T_x(y) = a \cd ((D_1 - i)^{-1} \ot 1)(\xi \ot y)=
a \cd ((D_1\ot 1 - i)^{-1} )(\xi \ot y)$.

By Theorem \ref{t:extright} we have $\xi \ot y \in \sD(1 \ot_{\Na^0} D_2)$
because $\xi \in X_1$ and $y \in \sD(D_2)$. Since $(X_1,\Na^0)$ is a correspondence 
we infer that $((D_1 - i)^{-1} \ot 1)(\sD(1 \ot_{\Na^0} D_2)) \su \sD(1 \ot_{\Na^0} D_2)$. 
Finally, from Def. \ref{def:Correspondence} (3) we conclude 
$a \cd ((D_1\ot 1 - i)^{-1} )(\xi \ot y) \in \sD(1 \ot_{\Na^0} D_2)$.

Now, let $z \in \sD(1 \ot_{\Na^0} D_2)$. We will show that then already
$T_x^*(z) \in \sD(D_2)$. Since $x = a \cd (D_1 - i)^{-1}(\xi)$ we have
that $T_x^*(z) = T_\xi^* (D_1 \ot 1 + i)^{-1} a^*(z)=:T_\xi^*\tilde z$. 
However, again, since $(X_1,\Na^0)$ is a correspondence we get that 
$\tilde z \in \sD(D_1 \ot 1) \cap \sD(1 \ot_{\Na^0} D_2)$. 
Now $\xi = \{b_i\} \in X_1 = PB_1^\infty$ and thus $T_\xi^*(\tilde z) = \xi^* 
\cd \tilde z \in Y$ where $\xi^* \in B_1 \ot \cK$ denotes the infinite row obtained as
the adjoint of the infinite column $\xi = \{b_i\}$. The result of the lemma now
follows from Lemma \ref{commcompact} since $z \in \sD(1 \ot_{\Na^0} D_2) \su
\sD(\diag(D_2))$.
\end{proof}

The next lemma implies that the first condition in Theorem \ref{t:kuceodd} on
the boundedness of the commutator is satisfied for our interior unbounded
product $(D,(X \hot_B Y)^2)$. Remark that the commutator is well--defined by
Lemma \ref{domaininclus}.

\begin{lemma}\label{l:Cond1}
The commutator
\[
\begin{split}
& \Big[  \matr{cc}{
D_1 \ot 1 \pm i\, 1 \ot_{\Na^0} D_2 & 0 \\
0 & \pm i D_2
}, \matr{cc}{
0 & T_x \\
T_x^* & 0
}\Big] \\
& \q : \big( \sD(D_1 \ot 1) \cap \sD(1 \ot_{\Na^0} D_2)\big) \op \sD(D_2)
\to (X \hot_B Y) \op Y
\end{split}
\]
extends to a bounded operator for all $x \in \sF$.
\end{lemma}
\begin{proof}
Let $x = a \cd (D_1 -i)^{-1}(\xi)$ with $\xi \in X_1$ and $a \in \sA$.
Then the identity
\[
\begin{split}
(D_1 \ot 1) T_x 
& = (D_1 a (D_1 -i)^{-1} \ot 1) T_\xi \\
& = ([D_1,a] (D_1 - i)^{-1} \ot 1) T_\xi + (a D_1(D_1 -i)^{-1} \ot 1) T_\xi
\end{split}
\]
proves that the operator $(D_1 \ot 1) T_x : \sD(D_2) \to X \hot_B
Y$ extends to a bounded operator. A similar calculation shows that $T_x^* (D_1
\ot 1) : \sD(D_1 \ot 1) \cap \sD(1 \ot_{\Na^0} D_2) \to Y$ extends to a
bounded operator $X \hot_B Y \to Y$. We therefore only need to prove that the
commutator
\[
\begin{split}
& \Big[  \matr{cc}{
1 \ot_{\Na^0} D_2 & 0 \\
0 & D_2
}, \matr{cc}{
0 & T_x \\
T_x^* & 0
}\Big] \\
& \q : \big( \sD(D_1 \ot 1) \cap \sD(1 \ot_{\Na^0} D_2)\big) \op \sD(D_2)
\to (X \hot_B Y) \op Y
\end{split}
\]
extends to a bounded operator.

This is equivalent to the boundedness of the two operators
\begin{align}
    (1 \ot_{\Na^0} D_2) T_x - T_x D_2 :& \sD(D_2) \to X \hot_B Y, \label{eq:firstcomm} \\
T_x^* (1 \ot_{\Na^0} D_2) - D_2 T_x^* :& \sD(D_1 \ot 1) \cap \sD(1 \ot_{\Na^0}
D_2) \to Y.
          \label{eq:seccomm}
\end{align}

To prove the boundedness of \eqref{eq:firstcomm} we calculate
\begin{equation}\label{eq:nabcommid}
\begin{split}
(1 &\ot_{\Na^0} D_2) T_x =
(1 \ot_{\Na^0} D_2) a ((D_1 -i)^{-1} \ot 1) T_\xi \\
& =
[1 \ot_{\Na^0} D_2, a] ((D_1 -i)^{-1} \ot 1) T_\xi
+ a [1 \ot_{\Na^0} D_2,(D_1 -i)^{-1} \ot 1] T_\xi \\
& \qq + a ((D_1 -i)^{-1} \ot 1)(1 \ot_{\Na^0} D_2)T_\xi.
\end{split}
\end{equation}
Since $(X_1,\Na^0)$ is a correspondence the first two summands on the rhs are bounded and
we can thus restrict our attention to the unbounded operator 
$(1 \ot_{\Na^0} D_2)T_\xi - T_\xi D_2 : \sD(D_2) \to X \hot_B Y$. 
However, by Theorem \ref{t:extright} the latter equals
$c \ci T_{\Na^0(\xi)}$ where $c : X \hot_B \B(Y) \hot_B Y \to X \hot_B Y$ 
is the evaluation map. This proves that the commutator in \eqref{eq:firstcomm} 
extends to a bounded operator.

The boundedness of \eqref{eq:seccomm} would follow from that of
\eqref{eq:firstcomm} if we knew that \eqref{eq:firstcomm} is
\emph{adjointable}. Since this is not established yet we need to prove
the boundedness of \eqref{eq:seccomm} separately.
By a computation similar to the one carried out
in \eqref{eq:nabcommid} we get that it suffices to prove that the unbounded
operator $T_\xi^* (1 \ot_{\Na^0} D_2) - D_2 T_\xi^* : \sD(D_1 \ot 1) \cap
\sD(1 \ot_{\Na^0} D_2) \to Y$ extends to a bounded operator. Furthermore, by
Proposition \ref{p:compself} we may replace the connection $\Na^0$ by the
Gra{\ss}mann $D_2$--connection $\Gc_{D_2}$. We then have the identity
\begin{equation}\label{}
T_{\xi}^* (1 \ot_{\Gc} D_2) - D_2 T_\xi^* = T_{\xi}^* \diag(D_2)
- D_2 T_\xi^*.
\end{equation}
Notice that we think of $\xi \in X_1 \su B_1^\infty$ as an element of
$B_1^\infty$ in the last identity. Furthermore, we are suppressing the
inclusion $\sD(1 \ot_{\Gc} D_2) \su \sD(\diag(D_2))$. But the right
hand side of \eqref{eq:seccomm} is bounded by Lemma \ref{commcompact}
and the proof is complete.
\end{proof}

We are now ready to state the main theorem of this paper. 

\begin{theorem}\label{t:prodcoincide}
Let $(X,D_1)$ and $(Y,D_2)$ be two odd unbounded Kasparov modules for $(A,B)$
and $(B,C)$ respectively. Suppose that there exists a correspondence
$(X_1,\Na^0)$ from $(X,D_1)$ to $(Y,D_2)$. Let $\Na : X_1 \to X \hot_B \B(Y)$ 
be any completely bounded hermitian $D_2$--connection. Then the even
unbounded Kasparov $A$--$C$ module $((X \hot_B Y)^2,D_1 \ti_{\Na} D_2)$
represents the Kasparov product of $(X,D_1)$ and $(Y,D_2)$.
\end{theorem}
\begin{proof} $((X\hot_B Y)^2,D_1 \ti_{\Na^0} D_2)$ is an even unbounded Kasparov module by
Theorem \ref{t:prodmod}. For any even unbounded Kasparov module
$(D,Z)$ and any odd selfadjoint operator $R \in \B(X)$ the perturbed
unbounded Kasparov module $(D + R,Z)$ gives rise to the same class in 
$KK$--theory under the bounded transform. Therefore, by 
Proposition \ref{p:compself} it suffices to prove the Theorem for $\Na=\Na^0$.

We apply Theorem \ref{t:kuceodd}. Its first condition is satisfied by
Lemma \ref{l:Cond1}. The second condition in
Theorem \ref{t:kuceodd} is satisfied since 
$\sD(D_1 \ti_{\Na^0} D_2) = (\sD(D_1 \ot 1) \cap \sD(1 \ot_{\Na^0} D_2))^2$.
Finally, the third condition in Theorem \ref{t:kuceodd} follows from 
\cite[Lemma 7.6]{KaaLes:LGP}, which implies that there
exists a constant $C > 0$ such that 
\[
\begin{split}
& \binn{(D_1 \ot 1 \pm i\, 1 \ot_{\Na^0} D_2)(z), 
(D_1 \ot 1 \pm i\, 1 \ot_{\Na^0} D_2)(z)} \\
& \q \geq
\frac{1}{2}\inn{(D_1 \ot 1)(z),(D_1 \ot 1)(z)} 
+ \inn{(1 \ot_{\Na^0} D_2)(z),(1 \ot_{\Na^0} D_2)(z)} - C\inn{z,z}
\end{split}
\]
for all $z \in \sD(D_1 \ot 1) \cap \sD(1 \ot_{\Na^0}D_2)$.
\end{proof}

\pagebreak[3]
\section{Application: Dirac--Schr\"odinger operators on complete manifolds}\label{s:GE}   

\subsection{Standing assumptions}\label{ss:GESA} Let $M^m$ be a complete oriented Riemannian
manifold $M$ of dimension $m$ and let $H$ be a separable Hilbert space. We thus have
the operator $*$--module $C^1_0(M,H)$ over the operator $*$--algebra
$C^1_0(M)$. This is a consequence of Proposition \ref{conestar} and Proposition
\ref{conemodu}. The operator $*$--algebra $C^1_0(M)$ sits as a dense
$*$--subalgebra inside the $C^*$--algebra $C_0(M)$ of continuous functions
vanishing at infinity and the inclusion is completely bounded. Thus with the
pair $(X=C^1_0(M,H),C^1_0(M))$ we are in the situation of Convention
\ref{StaConA}. 
Remark that the completeness of the manifold
entails that the operator $*$--algebra $C^1_0(M)$ is $\si$-unital, see Note
\ref{n:AppUniCom}.

Furthermore, let $D_{2,0} : \Ga_c^\infty(M,F) \to L^2(M,F)$ be a first order
elliptic differential operator acting on the sections of the smooth hermitian
vector bundle $F \to M$ over $M$.  We assume that $D_{2,0}$ is symmetric with
respect to the scalar product of $L^2(M,F)$ and that $D_{2,0}$ has
\emph{bounded propagation speed}, that is the symbol, 
$\sigma_{D_2} : T^*M \to \End(F)$ satisfies
\begin{equation}\label{eq:BPS}  
\sup_{\xi\in T^*M, \|\xi\|\le 1}
\|\sigma_{D_2}(\xi)\|=:C_{\textup{ps}}<\infty.
\end{equation}
By the classical Theorem of Chernoff \cite{Che:ESA} the completeness of
$M$ together with the bounded propagation speed assumption imply
the essential selfadjointness of $D_{2,0}$. By $D_2$ we then
denote its selfadjoint closure.

The set-up outlined in this Subsection \ref{ss:GESA} will be in effect
during the remainder of this Section \ref{s:GE}.

\subsection{Hermitian $D_2$--connections}\label{geomdelta}

We shall now see that the composition of the exterior differential and the symbol of the
first order differential operator $D_2$ is an example of a hermitian
$D_2$--connection. In fact we will interpret this composition as a
Gra{\ss}mann $D_2$--connection.

The following small lemma, which should be well-known
will be useful for proving complete boundedness of
the commutator $[D_2,\cd]$.

\begin{lemma}\label{l:autocomp}
Let $V$ be a $N$--dimensional Hilbert space and let $Z$ be an operator
space. Then any linear map $\al : V \to Z$ is completely bounded with
$\|\al\|_{\cb} \leq N \cd \|\al\|$.
\end{lemma}
\begin{proof}
We remark that $V$ has the structure of an operator space using the matrix norm
$\|\xi\| = \|\inn{\xi,\xi}^{1/2}\|_{\cc}$, $\xi \in M(V)$. Let
$\{e_i\}_{i=1}^N$ be an orthonormal basis for $V$ and let $\xi \in M_n(V)$ be
an $(n \ti n)$--matrix. We can then write the matrix $\xi$ as the sum $\xi =
\sum_{i=1}^N \xi_i \cd e_i$ for some unique matrices $\xi_i \in M_n(\cc)$. In
particular, we get the inequalities
\[
\| \al(\xi) \|_Z = \| \sum_{i=1}^N \xi_i \cd \al(e_i)\|_Z 
\leq \sum_{i=1}^N \|\xi_i\|_{\cc} \|\al(e_i)\|_Z
\leq N \cd \|\xi\|_V \cd \|\al\|,
\]
which in turn prove the lemma.
\end{proof}

\begin{prop}\label{p:symbext} 
The symbol of $D_2$ determines a completely bounded operator
\[
\si_{D_2} : \Ga_0(T^* M) \to \B\big(L^2(M,F) \big).
\]
Here $\Ga_0(T^*M)$ denotes the Hilbert $C_0(M)$--module of continuous one--forms
which vanish at infinity.
\end{prop}
\begin{proof}
We remark that $\Ga_0(T^*M)$ is an operator space when equipped with the matrix
norm $\|\om\| := \T{sup}_{x \in M}\|\inn{\om(x),\om(x)}\|^{1/2}_\cc$, $\om \in
M(\Ga_0(T^*M))$. The symbol of $D_2$ defines a linear map $\si_{D_2} : T_x^* M
\to \B(F_x)$ for each $x \in M$. The bounded propagation speed assumption
 \Eqref{eq:BPS} together with Lemma \ref{l:autocomp} then implies
 for each $\om \in M(\Ga_0(T^*M))$
\[
\| \si_{D_2}(\om) \| 
\leq \sup_{x \in M}\|\si_{D_2}(\om(x))\|_{\B(F_x)}
\leq  m\cd C_{\textup{ps}} \cd \sup_{x \in M}\|\om(x)\|_{T_x^*M} = m \cd C_{\textup{ps}} \cd \|\om\|,
\]
hence the claim.
\end{proof}

\begin{cor}\label{cor:D2CB}
The commutator with $D_2$ determines a completely bounded map $[D_2,\cd]
: C^1_0(M) \to \B(L^2(M,F))$.
\end{cor}
\begin{proof}
Since $D_2$ is a first order operator we have
$[D_2,f] = \si_{D_2}(df)$ for each $f \in C^1_0(M)$.
The result then follows from Proposition
\ref{p:symbext} since $d : C^1_0(M) \to \Ga_0(T^*M)$ 
is completely bounded by construction, \cf Remark \ref{r:CompleteInclusion}, 1. 
\end{proof}

We emphasize that the operator $*$-algebra structure on $C^1_0(M)$ comes from
the exterior derivative $d$ and not from $D_2$. Thus the above Corollary
is not immediate from Remark \ref{r:CompleteInclusion}. Rather
the bounded propagation speed assumption enters crucially in
Prop. \ref{p:symbext}.

We remark that the left action $C_0(M) \to \B(\Om^1_{D_2})$, \cf Def.
\ref{def:DForms} and the paragraph thereafter, is
essential in this case. It thus follows from the above results that we have
the Gra{\ss}mann $D_2$--connection
\begin{equation}\label{eq:GEGraCon}  
\Gc_{D_2}: C^1_0(M,H) \to C_0(M,H) \hot_{C_0(M)} \B(L^2(M,F)).
\end{equation}
It is not hard to see that $\Gc_{D_2}$ coincides with the composition
of the exterior differential and the symbol of $D_2$. In particular this
composition is a hermitian $D_2$--connection by Proposition
\ref{p:LCHerm}.

\subsection{Dirac--Schr\"odinger operators}\label{ss:exdirschr}

In addition to the Standing Assumptions \ref{ss:GESA} we assume that
$W$ is another Hilbert space which is continuously and densely
embedded in $H$ such that the inclusion map $W\hookrightarrow H$
is compact. Fix a family of selfadjoint operators
$\{D_1(x)\}_{x \in M}$ parametrized by the manifold $M$
such that the following conditions are satisfied:
\begin{enumerate}\renewcommand{\labelenumi}{\textup{(A \arabic{enumi})}}
\item\label{Ass1} 
The map $D_1 : M \to \B(W,H)$ is weakly differentiable. This means that the map
  $x \mapsto \inn{D_1(x)\xi,\eta}$ is differentiable for all $\xi \in W$ and
  $\eta \in H$. Furthermore, we suppose that the weak derivative $d(D_1)(x) :
  W \to H \ot T_x^*(M)$ is bounded for each $x \in M$ and that the supremum
  $\sup_{x \in M}\|d(D_1)(x)\| =: K < \infty$ is finite.

\item\label{Ass2} The domain, $\sD(D_1(x))=W$, is independent of $x\in M$
and equals $W$. Moreover, there exist constants
  $C_1,C_2 > 0$ such that
\begin{equation}\label{eq:Ass2}
C_1 \|\xi\|_{W} \leq \|\xi\|_{D_1(x)} \leq C_2 \|\xi\|_{W}
\end{equation}
for all $\xi \in W$ and all $x \in M$. Thus the graph norms are uniformly
equivalent to the norm $\|\cd\|_W$ of $W$.
\end{enumerate} 
\newcommand{\Aref}[1]{\textup{(A \ref{#1})}}
\newcommand{\AssA}{\Aref{Ass1}}
\newcommand{\AssB}{\Aref{Ass2}}
\begin{remark}\label{r:DS}
1. These assumptions correspond to the assumptions
(A-1), (A-2) of \cite{RobSal:SFM} in the one-dimensional case. When comparing,
note that in our \Aref{Ass2} it suffices to assume the first inequality in
\Eqref{eq:Ass2}. The second then follows from the Closed Graph Theorem
and the assumption $\sD(D_1(x))=W$.

2. \Aref{Ass1} implies that $M\ni x\mapsto D_1(x)\in \B(W,H)$
is a continuous map from $M$ into the bounded linear operators $W\to H$. 
To see this  we note that for $\xi\in W, \eta\in H$ and $x,y$ 
in a geodesic coordinate system, with $\gamma(x,y)$ denoting the unique
shortest smooth path from $x$ to $y$, 
\begin{equation}\label{eq:5}  
\begin{split}
    \Bigl| \inn{(D_1(x)-D_1(y)) \xi, \eta} \Bigr|
      &=  \Bigl| \int_{\gamma(x,y)}\inn{d(D_1(s))\xi, \eta}\Bigr|   \\
      & \le \sup_{s \in M} \|d(D_1)(s)\| 
            \cd \dist(x,y) \cd \|\xi \|_W \cd \|\eta\|_H,
\end{split}    
\end{equation}
hence at least locally $\|D_1(x)-D_1(y)\|\le K \cd \dist (x,y)$.

The assumption \Aref{Ass2} clearly implies that the supremum $\T{sup}_{x \in M}\|D_1(x)\|$ is finite.

As a consequence of these observations we get that the assignment $D_1(f)(x) :=
D_1(x)(f(x))$ defines a bounded operator $D_1 : C_0(M,W) \to C_0(M,H)$,
which may also be viewed as an unbounded operator in $C_0(M,H)$ with
domain $C_0(M,W)$. It is not hard to verify that our conditions on
the family $\{D_1(x)\}_{x \in M}$ imply that $D_1$ is a well-defined
selfadjoint and regular operator, \cf \cite[Theorem 4.2, 2. and Theorem 5.8]{KaaLes:LGP}.
\end{remark}

%

We remind the reader of the Standing Assumptions \ref{ss:GESA}
and the Gra{\ss}mann $D_2$--connection $\Gc_{D_2}$, \Eqref{eq:GEGraCon}. 
We have identifications
\begin{equation}\label{eq:Ident}
C_0(M,H) \hot_{C_0(M)} L^2(M,F) \cong (L^2(M,F))^\infty \cong L^2(M, H \ot F)  
       \cong H\ot  L^2(M,F)
\end{equation}
of Hilbert spaces and hence, by slight abuse of notation 
$1 \ot_{\Gc} D_2 = \diag(D_2) = D_2$. 
We notice that $\Ga_c^\infty(M,H \ot F)$ is a core for $D_2 = 1
\ot_{\Gc} D_2$.

We will use the notation $D_1(\cd) := D_1 \ot 1$ for the
selfadjoint operator on $L^2(M,H \ot F)$
associated to $D_1$. This notation is very suggestive since
for a function $f\in L^2(M,W\ot F)$ 
we have the pointwise identity $(D_1(\cd)f)(x) = D_1(x) f(x)$ for a.e. $x\in M$.

We are now going to prove the selfadjointness of the product operator $D_1
\ti_{\Gc} D_2$. To this end we need to verify the conditions in
Theorem \ref{t:regself}. The core is given by the smooth compactly supported
sections, $\sE := \Ga_c^\infty(M,H \ot F)$. We start with the first
condition.

We will use the notation $d(D_1(\cd)) : L^2(W \ot F) \to L^2(H \ot T^*M \ot F)$
for the exterior derivative which is defined fiberwise by $d(D_1)(x) \ot 1 : W
\ot F_x \to H \ot T_x^*M \ot F_x$. Furthermore, by slight abuse of
notation we let $\si_{D_2} : L^2(M, H\ot
T^*M \ot F) \to L^2(M, H \ot F)$ denote the map which is defined fiberwise by
$\xi \ot \om \ot \eta \mapsto \xi \ot \si_{D_2}(\om)(\eta)$. Both $d(D_1(\cd))$
and $\si_{D_2}$ are bounded operators by \AssA\ and Proposition
\ref{p:symbext}.
%

\begin{lemma}\label{domlocal}
Let $s \in \Ga_c^\infty(M,H \ot F)$ be a smooth compactly supported section. The
vector $(D_1(\cd) - i\cd \mu )^{-1}(s)$ then lies in the
domain of $D_2$. Furthermore we have for $\mu\in\Rstar $ the explicit formula
\[
\begin{split}
D_2(D_1(\cd) - i\cd\mu)^{-1}(s)
     & = (D_1(\cd) - i\cd\mu )^{-1} D_2(s) \\ 
& \q - (D_1(\cd) - i\cd\mu )^{-1} \si_{D_2}\bigl( d(D_1(\cd)) \bigr)
             (D_1(\cd) - i\cd\mu )^{-1}(s)
\end{split}
\]
In particular we also have that $D_2(D_1(\cd) - i\cd\mu)^{-1}(s) \in \sD(D_1(\cd))$.
\end{lemma}
\begin{proof}
Let us consider a smooth compactly supported section $t \in \Ga_c^\infty(M,H
\ot F)$ with support contained in a single coordinate patch $U \su M$ with
coordinates $(x_1,\ldots,x_m)$. 

We start by noting that the function $x \mapsto \inn{(D_1(x) - i \cd \mu)^{-1}s(x),t(x)}$ is differentiable with partial derivatives given by
\[
\begin{split}
& \frac{\pa}{\pa x_j}\inn{(D_1(x) - i \cd \mu)^{-1}s(x),t(x)} \\
& \q =
\binn{(D_1(x) - i \cd \mu)^{-1}\frac{\pa s(x)}{\pa x_j},t(x)}
+ \binn{(D_1(x) - i \cd \mu)^{-1}s(x),\frac{\pa t(x)}{\pa x_j}} \\
& \qq -
\binn{(D_1(x) - i \cd \mu)^{-1}\frac{\pa D_1(x)}{\pa x_j}
(D_1(x) - i \cd \mu)^{-1} t(x),s(x)}.
\end{split}
\]

Now, suppose that $D_2$ is given by the local formula $\sum_{j=1}^m A_j
\frac{\pa}{\pa x_j} + B$ over $U$. Using the above computation, we then get
that
\[
\begin{split}
 \int_M& \inn{(D_1(x) - i \cd \mu)^{-1}s(x),D_2 t (x)}\, d\T{vol} \\
        = & \sum_j \int_M \inn{(D_1(x) - i \cd \mu)^{-1}A_j(x)^*s(x), 
            \frac{\pa t(x)}{\pa x_j}} \,d\T{vol} \\
       & \q +\int_M \inn{(D_1(x) - i \cd \mu)^{-1}B(x)^*s(x),t (x)}\, d\T{vol}
      \end{split}
\]
\[
\begin{split}
   \q =& \int_M \inn{(D_1(x) - i \cd \mu)^{-1}D_2s(x),t(x)} \, d\T{vol} \\
       & \q 
+ \sum_j \int_M \inn{(D_1(x) - i \cd \mu)^{-1}\frac{\pa D_1(x)}{\pa x_j}
(D_1(x) - i \cd \mu)^{-1} A_j(x)^*s(x), t(x)} \, d\T{vol} \\
     = & \int_M \inn{(D_1(x) - i \cd \mu)^{-1}D_2s(x),t(x)} \, d\T{vol} \\
    & \q - \int_M \inn{(D_1(x) - i \cd \mu)^{-1}\si_{D_2}(d D_1(x))
(D_1(x) - i \cd \mu)^{-1}s(x), t(x)} \, d\T{vol}.
\end{split}
\]
The claim of the lemma now follows by a partition of unity argument.
\end{proof}

\begin{theorem}\label{t:DSSA} Under the standing assumptions
and \AssA, \AssB\ the Dirac--Schr\"odinger operator
\[
\matr{cc}{ 0 & D_1(\cd) - i D_2 \\
D_1(\cd) + i D_2 & 0 } \\
: \big( \sD(D_1(\cd)) \cap \sD(D_2) \big)^2 \to L^2(M,H \ot F)^2
\]
associated with the family of unbounded operators $\{D_1(x)\}_{x\in M}$ and the
differential operator $D_2$ agrees with the unbounded product operator $D_1
\ti_{\Gc} D_2$ and it is selfadjoint.
\end{theorem}
Note that for operators in Hilbert spaces (\textit{i.e.,} Hilbert $C^*$--modules over $\C$)
regularity is not an issue. Therefore, $D_1\ti_{\Gc} D_2$ is automatically regular.
\begin{proof}
By Theorem \ref{t:regself} and Lemma \ref{domlocal} we only need to prove that
the operator $[D_1(\cd),D_2](D_1(\cd) - i\cd\mu)^{-1} : \Ga_c^\infty(H \ot F) \to L^2(H \ot F)$ extends to a bounded operator. Now, by an application of Lemma \ref{domlocal} we have that
\[
\begin{split}
[D_1(\cd), D_2](D_1(\cd) - i\cd\mu)^{-1} 
& = (D_1(\cd) - i \cd \mu) D_2 (D_1(\cd)-i\cd \mu)^{-1}- D_2 \\
& = -\si_{D_2}d(D_1(\cd))(D_1(\cd) - i\cd\mu)^{-1}
\end{split}
\]
and the desired boundedness result follows since $d(D_1(\cd)) : L^2(W \ot F)
\to L^2(H \ot T^*M \ot F)$ and $\si_{D_2} : L^2(H \ot T^*M \ot F) \to L^2(H \ot
F)$ are bounded operators.
\end{proof}


\subsection{The index of Dirac--Schr\"odinger operators on complete manifolds}
We continue in the setting of the Standing Assumptions \ref{ss:GESA} and 
\AssA, \AssB. On top of these conditions we require
\begin{enumerate}\renewcommand{\labelenumi}{\textup{(A \arabic{enumi})}}
\setcounter{enumi}{2}
\item\label{Ass3} 
that there exist a compact set $K \su M$ and a constant $c>0$ such that
the spectrum $\spec(D_1(x)) \su (-\infty,-c] \cup [c,\infty)$ is uniformly
bounded away from zero for all $x \in M\setminus K$. 
\end{enumerate}
This condition corresponds to (A-3) in \cite{RobSal:SFM}, however we do not
assume that $D_1(x)$ has limits as $x$ approaches infinity.

The ellipticity of $D_2$ implies that the composition
\[
\begin{CD}
\sD(D_2) @>>> L^2(M,F) @>f>> L^2(M,F)
\end{CD}
\]
of the inclusion and multiplication with any $f \in C_0(M)$ is
\emph{compact}. This is immediate for compactly supported smooth $f$
and then follows since $C^\infty_c(M)$ is dense in $C_0(M)$.

It is then not hard to see that the conditions on our
differential operator $D_2$ imply that the pair $(D_2, L^2(M,F))$ 
is an odd unbounded Kasparov
$C_0(M)$--$\cc$ module. We let $[D_2]:= F(D_2, L^2(M,F)) \in
KK^1(C_0(M),\cc) \cong K^1(C_0(M))$ denote the odd $K$--homology class
obtained from the differential operator $D_2$ under the bounded transform,
\cf the beginning of Sec.~\ref{s:URI}.

We shall now see that the family $\{D_1(x)\}$ gives rise to an odd unbounded Kasparov
$\cc$--$C_0(M)$ module after a small modification.

\begin{prop}
Let $\psi \in C_0^1(M)$ be a $C^1$-function which vanishes at infinity
such that $\psi(x) > 0$ for all $x \in M$ and $\psi(x) = 1$ for all $x \in
K$. Then the family $\{\psi^{-1}(x) \cd D_1(x)\}_{x \in M}$ defines an odd
unbounded Kasparov $\cc$--$C_0(M)$ module $(\psi^{-1}\cd D_1, C_0(M,H))$.
\end{prop}
\begin{proof}
We define the unbounded operator
\[
\psi^{-1} \cd D_1 : \sD(\psi^{-1} \cd D_1) \to C_0(M,H),
 \quad (\psi^{-1} \cd D_1)(f)(x) = \psi^{-1}(x) \cd D_1(x)(f(x)).
\]
The domain is given by $\sD(\psi^{-1} \cd D_1) = \bigsetdef{f \in C_0(M,W)}{%
D_1(f) \in \psi \cd C_0(M,H)}$.

We start by proving that $\psi^{-1} \cd D_1$ is selfadjoint and regular. 
$\psi^{-1} \cd D_1$ is certainly symmetric. To see that it is closed we
let $\{f_n\}$ be a sequence in the domain such that $f_n\to f$
and $(\psi^{-1} \cd D_1)(f_n) \to g$ is convergent in $C_0(M,H)$. 
It follows that $\{D_1(f_n)\}$
is convergent. But $D_1$ is closed so $f \in \sD(D_1)$ with $D_1(f) = \psi \cd
g$. This proves that $\psi^{-1} D_1$ is closed. The selfadjointness and
regularity now follows by \cite[Theorem 4.2, 2. and Theorem 5.8]{KaaLes:LGP}.
Indeed, the localized unbounded operator at $x \in M$ is simply given by
$\psi^{-1}(x)\cd D_1(x) : W \to H$ which is selfadjoint by assumption.


Finally we show that the resolvent $(\psi^{-1} \cd D_1 - i)^{-1} = \psi \cd
(D_1 -i\psi)^{-1}$ is compact. To this end we recall that the compact operators
on the Hilbert $C^*$--module $C_0(M,H)$ are given by $\cK(C_0(M,H)) =
C_0(M,\cK(H))$. Now, the operator $(D_1(x) - i \psi(x))^{-1} \in \cK(H)$ is
compact for all $x \in M$ since the inclusion $W \to H$ is compact and it
depends continuously on the parameter $x \in M$ by Remark \ref{r:DS}, 2. and
the resolvent identity. Furthermore, in view of \Aref{Ass3} and the spectral
theorem for unbounded selfadjoint operators we get that $\T{sup}_{x\in
M}\|(D_1(x) - i \psi(x))^{-1}\| < \infty$. Altogether this implies that
$(\psi^{-1} \cd D_1 - i)^{-1}= \psi (D_1-i \psi)^{-1}$ lies in $C_0(M,\cK(H))$,
proving the claim.
\end{proof}

We will use the notation $[D_1] := F(C_0(M,H),\psi^{-1} \cd D_1) \in
KK^1(\cc,C_0(M)) \cong K_1(C_0(M))$ for the odd $K$-theory class obtained from
the parametrized family $\{\psi^{-1}(x) D_1(x)\}_{x \in M}$ under the bounded
transform. As the notation suggests, the class $[D_1]$ is independent of
the choice of function $\psi \in C_0^1(M)$ as long as $\psi(x) > 0$ for
all $x \in M$ and $\psi|_K = 1$.

\begin{lemma}
Let $\psi \T{ and } \phi \in C_0^1(M)$ be two strictly positive $C^1$-functions
which vanish at infinity and with $\psi|_K = 1 = \phi|_K$. The odd unbounded
Kasparov modules $(\psi^{-1} D_1 , C_0(M,H)) \T{ and } (\phi^{-1} D_1 ,
C_0(M,H))$ then represent the same class in $KK^1(\cc,C_0(M))$ under the
bounded transform.
\end{lemma}
\begin{proof}
This follows by noting that the difference $D_1 (\psi^2 + D_1^2)^{-1/2} -
D_1(\phi^2 + D_1^2)^{-1/2} \in C_0(M,\cK(H))$ of bounded transforms is a
compact operator. Indeed, we have that the difference
\begin{equation}\label{eq:pointdiff}
D_1(x)(\psi^2(x) + D_1(x)^2)^{-1/2} - D_1(x)(\phi^2(x) +
D_1^2(x))^{-1/2} \in \cK(H)
\end{equation}
is compact for all $x \in M$ since the function 
$t \mapsto t(\psi^2(x) + t^2)^{-1/2} - t(\phi^2(x) + t^2)^{-1/2}$ lies in $C_0(\rr)$. Furthermore by
\cite[Prop.~2.2]{Les:USF} and Remark \ref{r:DS}, 2. we get that the quantity in
\eqref{eq:pointdiff} depends continuously on the parameter $x \in M$. Finally,
the vanishing at infinity follows from \Aref{Ass3} and the spectral theorem for
unbounded selfadjoint operators.
\end{proof}

We can now make a sensible definition of spectral flow.

\begin{dfn}
By the \emph{spectral flow} of the family $\{D_1(x)\}_{x \in M}$ with respect to the
differential operator $D_2 : \Ga_c^\infty(M,F) \to L^2(M,F)$ we will
understand the integer given by the interior $KK$--product $[D_1]
\hot_{C_0(M)} [D_2] \in KK(\cc,\cc)$ under the identification $KK(\cc,\cc)
\cong \zz$. We will apply the notation $\SF(D_1,D_2) \in \zz$ for the
spectral flow.
\end{dfn}

We remark that the interior Kasparov product between odd $K$--theory and odd
$K$--homology can be identified with the index pairing \cite[Sec.~18.10]{Bla:KTO2Ed}.

In order to describe the above spectral flow as the index of an unbounded
Fredholm operator we need to construct a correspondence between the odd
unbounded Kasparov modules $(\psi^{-1} \cd D_1,C_0(M,H))$ and
$(D_2,L^2(M,F))$. To be able to deduce the commutator condition 
Def. \ref{def:Correspondence}, (4) for the pair $\psi\ii D_1,D_2$ from that for
the pair $D_1, D_2$ we have to assume additionally that $d \psi\ii$ 
is globally bounded. Let us first show that there are sufficiently many such functions.

\begin{lemma}\label{l:AppDistFunc}
Let $M$ be a complete Riemannian manifold, $K\subset M$ a compact subset.
Then there exists $\psi\in C^\infty(M)$ strictly positive and vanishing
at infinity such that
\begin{enumerate}
\item $\psi(x)=1$ for all $x\in K$,
\item $\sup_{x\in M} | d \psi\ii(x) |<\infty$.
\end{enumerate}
\end{lemma}
\begin{proof} Fix $x_0\in M$. Let $\phi\in C^\infty(M)$ be a smooth
approximation of the distance function $\dist(\cd,x_0)$ \cite[Sec. 3]{Gaf:CPH}
in the sense that $|\phi(x)-dist(x,x_0)|\le 1$, $|d\phi(x)|\le 2$ for
all $x\in M$. Furthermore, let $\varphi\in C^\infty_c(M)$ be a compactly
supported cut--off function with $\varphi(x)=1$ for $x$ in a neighborhood of $K$.
Then $\psi\ii(x):= \varrho(x)+(1-\varrho(x))(2+ \phi(x))$ does the job.
\end{proof}

%

\begin{prop} Let $\psi\in C^1_0(M)$ be a strictly positive $C^1$--function
vanishing at infinity such that $\psi\ii$ satisfies conditions
$\textup{(1)}$ and $\textup{(2)}$ of Lemma \plref{l:AppDistFunc}. Then
the pair $(C^1_0(M,H),\Gc_{D_2})$ given by the operator $*$--module
$C^1_0(M,H)$ over the operator $*$--algebra $C^1_0(M)$ and the Gra{\ss}mann
$D_2$--connection is a correspondence from $(\psi^{-1}D_1,C_0(M,H))$ to
$(D_2,L^2(M,F))$.
\end{prop}
\begin{proof}
Lemma \ref{l:AppDistFunc} ensures that $\psi$ with the desired properties exists.
Then the result is mainly a consequence of Theorem \ref{t:extright}, Lemma
\ref{domlocal} and the proof of Theorem \ref{t:DSSA}.

Note that
\begin{equation}
   [\psi\ii D_1, D_2] = -\sigma_{D_2}(d \psi\ii) D_1 + \psi\ii [D_1, D_2].
\end{equation}
Bounded propagation speed \Eqref{eq:BPS}, the global boundedness
of $d\psi\ii$ and the proof of Theorem \ref{t:DSSA} now show that
$[\psi\ii D_1(\cd), D_2]( D_1(\cd) - i \cd\mu )\ii$ 
extends to a bounded operator.
\end{proof}

The main result of this section now follows from Theorem
\ref{t:prodcoincide}. See also Theorem \ref{t:DSSA}.\mpar{CHECK crossrefs}

\begin{cor}
Under the assumptions of the previous Proposition 
the Dirac--Schr\"odinger operator
$\psi^{-1}D_1(\cd) + i D_2 : \sD(\psi^{-1}D_1(\cd)) \cap \sD(D_2) \to L^2(M,H \ot F)$
is an unbounded Fredholm operator and the index
\[
\ind(\psi^{-1} D_1(\cd) + i D_2) = \SF(D_1,D_2)
\]
coincides with the spectral flow of the family $\{D_1(x)\}$ with respect to
the differential operator $D_2$.
\end{cor}

\subsubsection{Proof of Theorem \ref{I:ThmB}} Finally, we make the link
to the index of Dirac--Schr\"odinger operators as in \cite[Prop. 1.4]{Ang:ICO}.
In loc.~cit. operators of the form $D+i\gl A$, with $D$ being a Dirac-type 
operator and $A$ being a selfadjoint bundle homomorphism, are considered.
$D$ corresponds to our $D_2$ and $A(x)$ corresponds to our $D_1(x)$ in
the special case of a finite-dimensional Hilbert space $H$. Then under assumptions
which are similar but slightly stronger than
our \Aref{Ass1}, \Aref{Ass2}, \Aref{Ass3} it is proved that $D+i \gl A$
is Fredholm for $\gl$ large enough.

In fact, it is not hard to see that under our assumptions
\Aref{Ass1}, \Aref{Ass2}, \Aref{Ass3} it follows from \cite[Lemma 7.6]{KaaLes:LGP}
that given $C>0$ there exists a $\gl_0=\gl_0(C)>0$ large enough such that
the operator $\phi D_1(\cd)+i D_2$ is Fredholm for any $C^1$-function
$\phi\in C^1(M)$ which satisfies $\sup_{x\in M} |d\phi(x)|\le C$ and
$\phi(x)\ge \gl_0$ for $x\in M\setminus K$.

It then follows from the stability of the Fredholm index under
deformations in the graph topology, \textit{cf.,~e.g.,}~\cite{CorLab:IIM},
that for any such $\phi$ the index of $\phi D_1(\cd) + i D_2$ coincides
with the spectral flow $\SF(D_1,D_2)$. This argument in particular
applies to the function $\phi\equiv\gl$ for $\gl\ge \gl_0(1)$.
\bibliography{mlbib.bib,localbib.bib} 

\newcommand{\etalchar}[1]{$^{#1}$}
\def\cprime{$'$}
\providecommand{\bysame}{\leavevmode\hbox to3em{\hrulefill}\thinspace}
\providecommand{\MR}{\relax\ifhmode\unskip\space\fi MR }
\providecommand{\MRhref}[2]{%
  \href{http://www.ams.org/mathscinet-getitem?mr=#1}{#2}
}
\providecommand{\href}[2]{#2}
\begin{thebibliography}{\textsc{GLM{\etalchar{+}}11}}

\bibitem[\textsc{Ang93a}]{Ang:ICO}
\textsc{N.~Anghel}, \emph{On the index of {C}allias-type operators}, Geom.
  Funct. Anal. \textbf{3} (1993), no.~5, 431--438. \MR{1233861 (94m:58213)}

\bibitem[\textsc{Ang93b}]{Ang:AIT}
\textsc{N.~Anghel}, \emph{An abstract index theorem on noncompact {R}iemannian
  manifolds}, Houston J. Math. \textbf{19} (1993), no.~2, 223--237. \MR{1225459
  (94c:58193)}

\bibitem[\textsc{BaJu83}]{BaaJul:TBK}
\textsc{S.~Baaj} and \textsc{P.~Julg}, \emph{Th\'eorie bivariante de {K}asparov
  et op\'erateurs non born\'es dans les {$C\sp{\ast} $}-modules hilbertiens},
  C. R. Acad. Sci. Paris S\'er. I Math. \textbf{296} (1983), no.~21, 875--878.
  \MR{715325 (84m:46091)}

\bibitem[\textsc{Bla98}]{Bla:KTO2Ed}
\textsc{B.~Blackadar}, \emph{{$K$}-theory for operator algebras}, second ed.,
  Mathematical Sciences Research Institute Publications, vol.~5, Cambridge
  University Press, Cambridge, 1998. \MR{1656031 (99g:46104)}

\bibitem[\textsc{Ble95}]{Ble:CBC}
\textsc{D.~P. Blecher}, \emph{A completely bounded characterization of operator
  algebras}, Math. Ann. \textbf{303} (1995), no.~2, 227--239. \MR{1348798
  (96k:46098)}

\bibitem[\textsc{Ble96}]{Ble:GHM}
\bysame, \emph{A generalization of {H}ilbert modules}, J. Funct. Anal.
  \textbf{136} (1996), no.~2, 365--421. \MR{1380659 (97g:46071)}

\bibitem[\textsc{Che73}]{Che:ESA}
\textsc{P.~R. Chernoff}, \emph{Essential self-adjointness of powers of
  generators of hyperbolic equations}, J. Functional Analysis \textbf{12}
  (1973), 401--414. \MR{0369890 (51 \#6119)}

\bibitem[\textsc{ChSi87}]{ChrSin:RCB}
\textsc{E.~Christensen} and \textsc{A.~M. Sinclair}, \emph{Representations of
  completely bounded multilinear operators}, J. Funct. Anal. \textbf{72}
  (1987), no.~1, 151--181. \MR{883506 (89f:46113)}

\bibitem[\textsc{ChSi89}]{ChrSin:SCB}
\bysame, \emph{A survey of completely bounded operators}, Bull. London Math.
  Soc. \textbf{21} (1989), no.~5, 417--448. \MR{1005819 (91b:46051)}

\bibitem[\textsc{CoLa63}]{CorLab:IIM}
\textsc{H.~O. Cordes} and \textsc{J.~P. Labrousse}, \emph{The invariance of the
  index in the metric space of closed operators}, J. Math. Mech. \textbf{12}
  (1963), 693--719. \MR{0162142 (28 \#5341)}

\bibitem[\textsc{Con80}]{Con:CAG}
\textsc{A.~Connes}, \emph{{$C\sp{\ast} $} alg\`ebres et g\'eom\'etrie
  diff\'erentielle}, C. R. Acad. Sci. Paris S\'er. A-B \textbf{290} (1980),
  no.~13, A599--A604. \MR{572645 (81c:46053)}

\bibitem[\textsc{Con94}]{Con:NG}
\bysame, \emph{Noncommutative geometry}, Academic Press Inc., San Diego, CA,
  1994. \MR{1303779 (95j:46063)}

\bibitem[\textsc{EfRu88}]{EffRua:ROB}
\textsc{E.~G. Effros} and \textsc{Z.-J. Ruan}, \emph{Representations of
  operator bimodules and their applications}, J. Operator Theory \textbf{19}
  (1988), no.~1, 137--158. \MR{950830 (91e:46077)}

\bibitem[\textsc{Gaf59}]{Gaf:CPH}
\textsc{M.~P. Gaffney}, \emph{The conservation property of the heat equation on
  {R}iemannian manifolds.}, Comm. Pure Appl. Math. \textbf{12} (1959), 1--11.
  \MR{0102097 (21 \#892)}

\bibitem[\textsc{GLM{\etalchar{+}}11}]{GesEtAl:IFS}
\textsc{F.~Gesztesy}, \textsc{Y.~Latushkin}, \textsc{K.~A. Makarov},
  \textsc{F.~Sukochev}, and \textsc{Y.~Tomilov}, \emph{The index formula and
  the spectral shift function for relatively trace class perturbations}, Adv.
  Math. \textbf{227} (2011), no.~1, 319--420. \MR{2782197}

\bibitem[\textsc{Hig06}]{Hig:RIT}
\textsc{N.~Higson}, \emph{The residue index theorem of {C}onnes and
  {M}oscovici}, Surveys in noncommutative geometry, Clay Math. Proc., vol.~6,
  Amer. Math. Soc., Providence, RI, 2006, pp.~71--126. \MR{2277669
  (2008b:58028)}

\bibitem[\textsc{Iva11}]{Iva:CBC}
\textsc{N.~P. Ivankov}, \emph{Completely bounded characterization of operator
  algebras with involution},  \texttt{arXiv:1104.2626v1 [math.OA]}.

\bibitem[\textsc{KaLe11}]{KaaLes:LGP}
\textsc{J.~Kaad} and \textsc{M.~Lesch}, \emph{A local global principle for
  regular operators in {H}ilbert {$C\sp{\ast} $}-modules},
  \texttt{arXiv:1107.2372 [math.OA]}.

\bibitem[\textsc{Kas80a}]{Kas:HCM}
\textsc{G.~G. Kasparov}, \emph{Hilbert {$C\sp{\ast} $}-modules: theorems of
  {S}tinespring and {V}oiculescu}, J. Operator Theory \textbf{4} (1980), no.~1,
  133--150. \MR{587371 (82b:46074)}

\bibitem[\textsc{Kas80b}]{Kas:OKF}
\bysame, \emph{The operator {$K$}-functor and extensions of {$C\sp{\ast}
  $}-algebras}, Izv. Akad. Nauk SSSR Ser. Mat. \textbf{44} (1980), no.~3,
  571--636, 719. \MR{582160 (81m:58075)}

\bibitem[\textsc{Kuc97}]{Kuc:KKP}
\textsc{D.~Kucerovsky}, \emph{The {$KK$}-product of unbounded modules},
  $K$-Theory \textbf{11} (1997), no.~1, 17--34. \MR{1435704 (98k:19007)}

\bibitem[\textsc{LaMi89}]{LawMic:SG}
\textsc{H.~B. Lawson, Jr.} and \textsc{M.-L. Michelsohn}, \emph{Spin geometry},
  Princeton Mathematical Series, vol.~38, Princeton University Press,
  Princeton, NJ, 1989. \MR{1031992 (91g:53001)}

\bibitem[\textsc{Lan95}]{Lan:HCM}
\textsc{E.~C. Lance}, \emph{Hilbert {$C\sp *$}-modules}, London Mathematical
  Society Lecture Note Series, vol. 210, Cambridge University Press, Cambridge,
  1995, A toolkit for operator algebraists. \MR{1325694 (96k:46100)}

\bibitem[\textsc{Les97}]{Les:OFT}
\textsc{M.~Lesch}, \emph{Operators of {F}uchs type, conical singularities, and
  asymptotic methods}, Teubner-Texte zur Mathematik [Teubner Texts in
  Mathematics], vol. 136, B. G. Teubner Verlagsgesellschaft mbH, Stuttgart,
  1997. \MR{1449639 (98d:58174)}

\bibitem[\textsc{Les05}]{Les:USF}
\bysame, \emph{The uniqueness of the spectral flow on spaces of unbounded
  self-adjoint {F}redholm operators}, Spectral geometry of manifolds with
  boundary and decomposition of manifolds, Contemp. Math., vol. 366, Amer.
  Math. Soc., Providence, RI, 2005, pp.~193--224. \texttt{arXiv:0401411
  [math.FA]}, \MR{2114489 (2005m:58049)}

\bibitem[\textsc{Mes09}]{Mes:UBK}
\textsc{B.~Mesland}, \emph{Unbounded bivariant {$K$}--theory and
  correspondences in noncommutative geometry},  \texttt{arXiv:0904.4383v2
  [math.KT]}.

\bibitem[\textsc{PaSm87}]{PauSmi:MMT}
\textsc{V.~I. Paulsen} and \textsc{R.~R. Smith}, \emph{Multilinear maps and
  tensor norms on operator systems}, J. Funct. Anal. \textbf{73} (1987), no.~2,
  258--276. \MR{899651 (89m:46099)}

\bibitem[\textsc{RoSa95}]{RobSal:SFM}
\textsc{J.~Robbin} and \textsc{D.~Salamon}, \emph{The spectral flow and the
  {M}aslov index}, Bull. London Math. Soc. \textbf{27} (1995), no.~1, 1--33.
  \MR{1331677 (96d:58021)}

\bibitem[\textsc{Rua88}]{Rua:SCA}
\textsc{Z.-J. Ruan}, \emph{Subspaces of {$C\sp *$}-algebras}, J. Funct. Anal.
  \textbf{76} (1988), no.~1, 217--230. \MR{923053 (89h:46082)}

\bibitem[\textsc{Wol73}]{Wol:ESA}
\textsc{J.~A. Wolf}, \emph{Essential self-adjointness for the {D}irac operator
  and its square}, Indiana Univ. Math. J. \textbf{22} (1972/73), 611--640.
  \MR{0311248 (46 \#10340)}

\bibitem[\textsc{Wor91}]{Wor:UEA}
\textsc{S.~L. Woronowicz}, \emph{Unbounded elements affiliated with {$C\sp
  *$}-algebras and noncompact quantum groups}, Comm. Math. Phys. \textbf{136}
  (1991), no.~2, 399--432. \MR{1096123 (92b:46117)}

\end{thebibliography}
\bibliographystyle{amsalpha-lmp}
\end{document}